\documentclass[pdflatex,sn-mathphys-num]{sn-jnl}

\usepackage{graphicx}%
\usepackage{multirow}%
\usepackage{amsmath,amssymb,amsfonts}%
\usepackage{amsthm}%
\usepackage{mathrsfs}%
\usepackage[title]{appendix}%
\usepackage{xcolor}%
\usepackage{textcomp}%
\usepackage{manyfoot}%
\usepackage{booktabs}%
\usepackage{algorithm}%
\usepackage{algorithmicx}%
\usepackage{algpseudocode}%
\usepackage{listings}%
\usepackage{tikz}
\usetikzlibrary{arrows, arrows.meta}
\usepackage{multicol}
\usepackage{booktabs}
\usepackage{pifont}

\theoremstyle{thmstyleone}%
\newtheorem{theorem}{Theorem}
\newtheorem{proposition}[theorem]{Proposition}%
\newtheorem{corollary}[theorem]{Corollary}%
\newtheorem{lemma}[theorem]{Lemma}%

\theoremstyle{thmstyletwo}%
\newtheorem{example}{Example}%
\newtheorem{remark}{Remark}%

\theoremstyle{thmstylethree}%
\newtheorem{definition}{Definition}%

\def\RR{\mathbb{R}}

\def\sii{\Leftrightarrow}

\def\II{I}
\def\LL{\mathcal{L}(\mathbb{R})}
\def\LLposit{\mathcal{L}(\mathbb{R}^+)}

\def\LR{\mathcal{L}(\RR)}

\def\colec{[\unda_1,\ova_1],\dots,[\unda_n,\ova_n]}
\newcommand{\reales}{\mathbb{R}}
\newcommand{\naturales}{\mathbb{N}}
\newcommand{\infII}{\mathbf{0}_{\II}}
\newcommand{\supII}{\mathbf{1}_{\II}}
\newcommand{\Ldinter}{\mathcal{L}(\II)}

\newcommand{\lator}{\preceq_{Lo}}
\newcommand{\ador}{\preceq_{ao}}
\newcommand{\lexunor}{\preceq_{Lex1}}
\newcommand{\lexdosor}{\preceq_{Lex2}}
\newcommand{\xuyor}{\preceq_{XY}}
\newcommand{\preab}{\preceq_{\alpha \beta}}

\newcommand{\unda}{\underline{a}}
\newcommand{\ova}{\overline{a}}
\newcommand{\undb}{\underline{b}}
\newcommand{\ovb}{\overline{b}}
\newcommand{\vecas}{(a_1,\dots,a_n)}

\newcommand{\maxto}{Max_{\preceq}}
\newcommand{\mini}{Min}
\newcommand{\maxi}{Max}
\newcommand{\infto}{Inf_{\preceq}}
\newcommand{\supto}{Sup_{\preceq}}
\newcommand{\inflo}{Inf_{Lo}}
\newcommand{\suplo}{Sup_{Lo}}
\newcommand{\minlexuno}{Min_{Lex1}}
\newcommand{\mincq}{Min_{\preceq}}

\newcommand{\setas}{\{a_1,\dots,a_n\}}
\newcommand{\al}{\alpha}

\newcommand{\terv}[3]{[#1_{#3},#2_{#3}]}

\definecolor{morado}{rgb}{0.3, 0, 0.7}

\definecolor{verdeos}{rgb}{0, .4, 0}

\newcommand{\cmark}{\color{green!40!gray}\ding{51}}%
\newcommand{\xmark}{\color{red}\ding{55}}%
\begin{document}

\title[On the formal definition of average function for interval data]{On the formal definition of average function for interval data}


\author[1]{\fnm{Sergio} \sur{F. Alonso}}\email{fernandezalosergio@uniovi.es}
\equalcont{These authors contributed equally to this work.}

\author[1]{\fnm{Susana} \sur{Montes}}\email{montes@uniovi.es}
\equalcont{These authors contributed equally to this work.}

\author*[1]{\fnm{Susana} \sur{Diaz-Vazquez}}\email{diazsusana@uniovi.es}
\equalcont{These authors contributed equally to this work.}

\affil*[1]{\orgdiv{Dept. Statistics and O. R. and M. T.}, \orgname{University of Oviedo}, \orgaddress{\street{Federico Garcia Lorca, 18}, \city{Oviedo}, \postcode{33007},  \country{Spain}}}


\abstract{
This paper addresses the formal definition of average functions for interval data, extending classical aggregation theory beyond precise numerical values to the more general setting of real-valued intervals. Recognizing that interval data naturally arise in contexts involving imprecision or uncertainty, we examine how established notions of averages and aggregation functions can be generalized when the underlying domain is no longer restricted to the unit interval, but instead encompasses the entire real line. We analyze the impact of different admissible orders on intervals, particularly focusing on the lattice order and several total orders, and provide rigorous definitions of idempotency and monotonicity in this context.  Furthermore, we adapt and study classical families of aggregation operators such as power means, minimum, maximum, and median to the interval setting, characterizing when these operators satisfy the properties required of average functions for intervals. Our results highlight that the choice of order plays a crucial role in the behavior and suitability of these operators. This work lays a foundation for the systematic aggregation of interval-valued information, with potential applications in statistics, decision making, and fields where data imprecision is inherent.}

\keywords{aggregation function, average function, interval data, separable function}



\maketitle

\section{Introduction}\label{sec1}

The aggregation of numbers is one of the most fundamental tools in information management. Its historical roots can be traced as far back as problem 14 of the Moscow papyrus, dated to 1850 BCE~\cite{archibald}, which some authors identify as the earliest written evidence of aggregation functions. However, the theoretical development of aggregation gained significant momentum in the twentieth century, driven by the advent of fuzzy logic~\cite{yager2004owa,zadeh1965fuzzy}, advances in computing, and the increasing need to process and summarize large datasets. This context gave rise to the challenge of condensing a dataset into a single, representative value, and to the ongoing effort to ensure that such summaries accurately reflect the underlying data.

Aggregation functions have long been the primary tools for summarizing datasets (see~\cite{beliakov2007aggregation,beliakov2016practical,grabisch2009}), and the study of their properties has been guided by the intended purpose of summarization specifically, by the level of representativeness desired. This notion of representativeness is context-dependent, and the nature of the problem at hand determines which features of the original datasets should be taken into account during aggregation.

On the other hand, the use of intervals to represent information has proven highly valuable in numerous contexts. See, for example,~\cite{dubois1988,moore1966} or~\cite{HroBriFil2017,OliAzePacVal2022} in the context of Symbolic Data Analysis).
In recent years, there has been a marked rise in research devoted to interval-based aggregation.
Researchers have proposed new operators with theoretical guarantees such as monotonicity, robustness, idempotency, and interpretability.
For instance, Diaz-Vazquez et al.~\cite{DiazVazquez2024new} introduced a novel family of interval aggregation functions  with superior theoretical and practical behavior.
Wilkin and Beliakov~\cite{WilkinBeliakov2022} developed robust methods based on compositional modes on the simplex for interval and compositional data. And Fumanal-Idocin et al.~\cite{Fumanal2021} proposed interval-valued aggregation operators based on moderate deviations, applied to brain computer interface tasks.
Zhang et al.~\cite{ZhangElastic2021} explored elastic statistical analysis of interval-valued time series, while Sun~\cite{Sun2016} provided a comprehensive overview of regression techniques specifically designed for interval-valued data
Such studies reflect the increasing demand for theoretically sound and practically effective tools to deal with interval-valued data in critical applications where reliability and interpretability matter.

However, to the best of our knowledge, there exists a notable gap in the existing literature regarding the formal analysis of averaging operators in the context of interval data aggregation. This work aims to address and fill this gap.

There are two approaches to considering intervals to represent information: the ontic and the epistemic ones~\cite{COUSO20141502}. From the epistemic point of view, the actual outcome we are trying to represent is a precise value, but for some reason this value cannot be determined in an accurate way and the best information that we have are bounds that determine an interval of possible right outcomes. From the ontic point of view, intervals themselves are the outcome, there is no imprecision in the sense that the whole set of values included in the interval is the (precise) information. In both cases the whole interval must be considered to depict the knowledge/opinion provided by the decision maker. Thus, our objective is to aggregate the interval data, this is, the data that, for various reasons such as measurement imprecision or inherent variability due to environmental conditions, are better represented as intervals. Although this work is theoretical in nature, we consider the study of interval aggregation to be broadly valuable given the wide applicability of intervals in contexts characterized by imprecision. Notable examples include interval-valued fuzzy sets~\cite{sambuc1975}, where the fundamental elements are intervals, and alpha-cuts of fuzzy numbers, in which intervals naturally arise.

In this contribution, we aim to define the concept of an average function for intervals. The notion of an average function for real numbers dates back to Cauchy in 1821~\cite{cauchy1821}, who required only that the function be internal; that is, he considered a function suitable for averaging if its output lies between the smallest and largest values among its inputs. In contemporary settings, it is common to require that an aggregation function be increasing, meaning that larger inputs yield larger outputs. Since our definition involves the notions of minimum and maximum, it will depend on the specific order imposed on the set of intervals. To ensure the greatest generality, we will consider admissible orders~\cite{Bustince} in our study. We will then characterize which interval functions qualify as average functions under these orders. Additionally, we introduce the concept of a separable interval function and use it to explore and analyze several properties of the defined functions.

The contribution is organized as follows: We recall some basic concepts and we set the notation in Section~\ref{sec:preliminaries}. Section~\ref{sec:average_number} is devoted to recalling the different approaches to the notion of an average function in the context of numbers. We discuss the same notion for interval data in Section~\ref{sec:average_interval}. Section~\ref{sec:classical_interval_averages} contains the study of the main particular types of averages for intervals. In Section~\ref{sec:example} we show a small example of the application of the operators introduced.
Finally, we address some conclusions in Section~\ref{sec:conclusion}.

\section{Preliminaries}\label{sec:preliminaries}

In this section, we establish the notation that is further used in this contribution.
We denote the closed (and bounded) intervals in $\reales$ by $a=[\unda,\ova]$, where obviously, $\unda,\ova \in \reales$ and $\unda \le \ova$. We denote by $\LL$ the set of closed intervals of real numbers, that is,
$\LL=\{[\unda,\ova]\mid \unda,\ova\in\RR\ and\ \unda\leq \ova\}$.

The notation $\LLposit$ refers to the set of intervals of non-negative values, which will be relevant in Section \ref{sec:classical_interval_averages}. Thus,
$
\LLposit=\{[\unda,\ova]\in \LL \mid 0 \le \unda\}.
$

Another particular case is commonly seen in the literature. Thus, intervals in the entire real line $\reales$ are not considered but subintervals of a prefixed interval in $\LL$. When this happens, the interval that replaces $\reales$ is denoted $\II$, its lower and upper endpoints are denoted $\infII$ and $\supII$, respectively, that is, $\II=[\infII,\supII]$, and $\Ldinter$ denotes the set of closed subintervals of $\II$, that is,
$$
\Ldinter=\{[\unda,\ova]\in \LL  \mid \infII \le \underline{a} \le \overline{a} \le \supII\}.
$$

Concerning arithmetic between intervals, we denote by $[\underline{a},\overline{a}]+[\underline{b},\overline{b}]$ the interval $[\underline{a}+\underline{b},\overline{a}+\overline{b}]$, that is, the interval containing $x+y$ for all $x\in[\unda,\ova]$ and $y\in[\underline{b},\overline{b}]$.\\
Concerning the product, $k[\underline{a},\overline{a}]=[k \underline{a},k\overline{a}]$ for $k \ge 0$ and $k[\underline{a},\overline{a}]=[k \overline{a},k\underline{a}]$ for $k < 0$, that is, the interval containing $kx$ for all $x\in[\unda,\ova]$.\\
When we consider natural numbers in this paper, we will understand that $0$ is not included: $\naturales=\{1, 2, \dots\}$.
\medskip

\begin{definition}
    Given a set $X$, a binary relation in $X$ is an order relation if it is reflexive, antisymmetric and transitive.
\end{definition}

From now on, we will consider the notation $\preceq$  for any order relation.

Clearly, there could be elements in the set that are not comparable under an order relation. In this case, it is said that the order is partial. Otherwise, it is said that it is total. More precisely,

\begin{definition}
    An order $\preceq$ defined on a set $X$ is said total if for every pair of elements $x,y\in X$ either $x\preceq y$ or $y\preceq x$ holds.
\end{definition}

\begin{definition}
  Given an order $\preceq$ defined on a set $X$ and two elements $x, y \in X$,
  \begin{itemize}
    \item their infimum (if it exists) is an element $inf$ such that $inf \preceq x,y$ and for any other element $z \in X$ such that $z \preceq x,y$ it also holds that $z \preceq inf$.
    \item their supremum (if it exists) is an element $sup$ such that $x,y \preceq sup$ and for any other element $z \in X$ such that $x,y \preceq z$, it also holds that $sup  \preceq z$.
  \end{itemize}
\end{definition}
As it is well known, if the order is total, there exist infimum and supremum for every pair of elements and they are one of them. However, the existence of these two bounds is not always guaranteed for partial orders. When they exist, a lattice is obtained.

\begin{definition}
Given a set $X$ and an order $\preceq$ on $X$, we say that $\preceq$ induces a lattice structure on $X$ (or that ($X, \preceq$) is a lattice) if for every pair of elements $x,y$ in $X$, there exists the infimum  and the supremum.
\end{definition}


It is well known that the natural order for real numbers (denoted $\leq$) is a total order, therefore, ($\reales, \leq$) is a lattice. However, when dealing with intervals, a problem arises because there is not a unique universally accepted order. The order that is considered the most intuitive one for intervals is the lattice order~\cite{Fishburn1987}, which has also been called product order~\cite{neggers1998basic}, coordinate order~\cite{davey2002introduction} or even component-wise order~\cite{taylor1999practical}. It is defined for a pair of intervals $[\unda,\ova], [\undb,\ovb]\in\LL$ as follows:
$$[\unda,\ova]\lator[\undb,\ovb]\quad \sii \quad \unda\leq\undb\ \quad and \quad  \ova\leq\ovb\,.$$

This order shows a significant disadvantage: it is not total, that is, not every pair of intervals is comparable according to it. For example, for the intervals $[2,50]$ and $[49,49.5]$, neither $[2,50]\not\lator[49,49.5]$ nor $[49,49.5]\not\lator[2,50]$, even in the case we could consider there is a clear order between then. This important drawback made necessary the search for other orders that do allow us to compare every pair of intervals.
In the following, we list three of the most common ones.

\begin{definition}\label{orders}
Let $[\unda,\ova], [\undb,\ovb]$ be two intervals in $\LL$.
\begin{itemize}
    \item The lexicographic order of type 1~\cite{Bustince} compares the lower bounds of the intervals and, if these are equal, it compares the upper bounds. Thus, it is defined as:

    $$[\unda,\ova]\lexunor[\undb,\ovb] \sii \unda<\undb\ or\ (\unda=\undb\ and \ \ova\leq \ovb).$$

    \item The lexicographic order of type 2~\cite{Bustince} compares the bounds of the intervals but, unlike type 1 lexicographic order, it gives priority to the upper bounds. Thus,

    $$[\unda,\ova]\lexdosor[\undb,\ovb] \sii \ova<\ovb\ or\ (\ova=\ovb\ and \ \unda\leq \undb)$$

    \item Xu and Yager order~\cite{Xu2006} compares the middle points of the intervals and, if this is not decisive, their lengths. It is defined as follows:
       $$
     [\unda,\ova]\xuyor[\undb,\ovb] \sii
      \unda+\ova<\undb+\ovb\ or \
(\unda+\ova=\undb+\ovb \ and \ \ova-\unda\le \ovb-\undb)
       $$
\end{itemize}
\end{definition}

Even though the lattice order is not a total order, it is the most intuitive one. This is the reason why, when a total order is necessary, other orders that respect the comparison method of the lattice order are considered. They are called admissible orders.

\begin{definition}
    ~\cite{Bustince} An admissible order $\preceq_{ao}$ is a total order that extends the lattice order, this is, that satisfies that if $[\unda,\ova]\lator[\undb,\ovb]$, then $[\unda,\ova]\ador[\undb,\ovb]$.
\end{definition}

As  examples of admissible orders we can consider the three orders of Definition~\ref{orders}.

In the same paper where the idea of admissible order was presented~\cite{Bustince}, the authors also proposed a way to obtain admissible orders that they called generated admissible orders.

\begin{definition}
    An admissible order $\preceq$ is said to be a generated admissible order if there exist two continuous functions $c_1, c_2: \LL^2\to\RR$ such that for every pair of intervals $[\unda,\ova],\ [\undb,\ovb]\in \mathcal{L}([0,1])$, we have that
    $[\unda,\ova]\preceq [\undb,\ovb]$ if, and only if, $[c_1(\unda,\ova),c_2(\unda,\ova)]\lexunor[c_1(\undb,\ovb),c_2(\undb,\ovb)]$.

\end{definition}

It is immediate that if we take $f$ as the projection onto the first component and $g$ as the projection onto the second component, we obtain the type 1 lexicographic order.
Similarly, if we take $f$ as the projection onto the second component and $g$ as the projection onto the first one, we obtain the type 2 lexicographic order, and if we take $f$ as the mean of the endpoints and $g$ as the length of the interval, we obtain the Xu and Yager order.

Moreover, the functions that allow us to generate these orders can be considered as particular cases of a family of functions firstly described in 1983 by Atanassov~\cite{Atanassov83} as $K_{\alpha}(x,y)=(1-\alpha)x+\alpha y$. If we set $K_{\alpha}$ and $K_{\beta}$
with $\alpha\neq\beta$, we get the generated admissible order given for two intervals $a, b\in\LL$ by

\begin{equation}\label{ordenesKalfa}
a  \preceq_{\alpha,\beta} b \Leftrightarrow
[K_{\alpha}(\unda,\ova),K_{\beta}(\unda,\ova)]\lexunor[K_{\alpha}(\undb,\ovb),K_{\beta}(\undb,\ovb)]
 \end{equation}

In the same manuscript, Bustince et al.~\cite{Bustince} also proved that for every $\alpha\in[0,1)$, all the orders $\preceq_{\alpha,\beta}$ with $\beta > \alpha$ are the same order and they denote them by $\preceq_{\alpha+}$.

Similarly, for every $\alpha\in(0,1]$, all the orders $\preceq_{\alpha\beta}$ with $\beta<\alpha$ are the same order and they are denoted as  $\preceq_{\alpha-}$.

Recall that for $\alpha=0$ we get the lexicographic order of type 1, for $\alpha=1$ we get the lexicographic order of type 2 while the Xu and Yager order is obtained for $\alpha=\frac{1}{2}, \beta> \alpha$:
$$
\lexunor  \equiv  \preceq_{0+} \qquad
\lexdosor \equiv  \preceq_{1-} \qquad
\xuyor    \equiv  \preceq_{\frac{1}{2}+}
$$

\section{Average functions for numbers}\label{sec:average_number}

This section presents an overview of the various approximations to the concept of an averaging function for numbers. Throughout this section, functions are not defined over all real numbers, but only on closed intervals in order to work with the boundary conditions imposed in~\cite{grabisch2009}. This is not our purpose in this paper,  and we will consider the general definition on $\mathbb{R}$ each time it is possible.

We first recall the definitions of internal, increasing and strictly increasing function.

\begin{definition} Let $\II$ be  a closed subinterval in $\RR$.
    A function $f:\II^n\to \II$ is said to be:
\begin{itemize}
\item internal if
    $\mini\{x_1,\dots,x_n\}\leq f(x_1,\dots,x_n)\leq \maxi\{x_1,\dots,x_n\}$ for every tuple $(x_1,\dots,x_n)\in \II^n$.
\item increasing if for every $\{x_1,\dots,x_n,y\} \subseteq I$ such that for some $i \in \{1,\dots,n\}$ it holds that $x_i < y$, then
    $$f(x_1,\dots x_i,\dots,x_n) \leq f(x_1,\dots,x_{i-1}, y,x_{i+1},\dots, x_n).$$

\item strictly increasing if for every $\{x_1,\dots,x_n,y\} \subseteq I$ such that for some $i \in \{1,\dots,n\}$ it holds that $x_i < y$, then
    $$f(x_1,\dots x_i,\dots,x_n) < f(x_1,\dots,x_{i-1}, y,x_{i+1},\dots, x_n).$$
\end{itemize}
\end{definition}

According to Cauchy, a function is an average if it is internal.
Cauchy did not require monotonicity in its proposal for the notion of average function.  We next show an example of an average function according to the Cauchy's definition, which does not have the expected property of monotonicity.  The operator $f:[0,1]^2\to[0,1]$ given by

$$
f(x,y)=\left\{\begin{array}{ll}
 \maxi\{x,y\} & \mbox{ if } (x,y) \in ([0,0.5]\times [0,0.5])\cup ( [0.5,1] \times [0.5,1])\\
\mini\{x,y\} & \mbox{ otherwise.}
\end{array}\right.
$$
is represented in Figure~\ref{Fig: 1}. It is immediate that $f$
is internal, since it is always smaller than or equal to the maximum and greater than or equal to the minimum.

\begin{figure}[H]
    \centering
\begin{tikzpicture}[xscale=3,yscale=3]

                \draw [black,ultra thick,arrows=->] (0,0) -- coordinate (x axis mid) (1.2,0);
    \draw [black,ultra thick,arrows=->] (0,0) -- coordinate (y axis mid) (0,1.2);

\draw [green,fill,opacity = .3] (0,0)--(0.5,0)--(0.5,0.5)--(0,0.5)--(0,0);
\draw [green,fill,opacity = .3] (0.5,0.5)--(0.5,1)--(1,1)--(1,.5)--(0.5,0.5);
\draw [green,fill,opacity = .05] (0.5,0)--(1,0)--(1,0.5)--(0.5,0.5)--(0.5,0);
\draw [green,fill,opacity = .05] (0,0.5)--(0.5,0.5)--(0.5,1)--(0,1)--(0,0.5);

\draw [black,thick] (0,.5)--(1,.5);
\draw [black,thick] (.5,0)--(.5,1);
\draw [black,thick] (1,0)--(1,1);
\draw [black,thick] (0,1)--(1,1);

\node[black] at (0.25,0.25) {\Large $Max$};
\node[black] at (0.75,0.25) {\Large $Min$};
\node[black] at (0.25,0.75) {\Large $Min$};
\node[black] at (0.75,0.75) {\Large $Max$};

\node[black] at (-0.1,-0.1) {$(0,0)$};
\node[black] at (1,-0.1) {$(1,0)$};
\node[black] at (-0.15,1.07) {$(0,1)$};
\node[black] at (1,1.07) {$(1,1)$};

\end{tikzpicture}
\caption{Example of an averaging function according to Cauchy}
    \label{Fig: 1}
\end{figure}
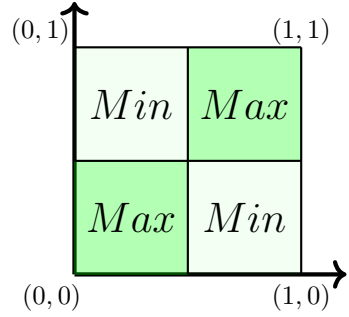

However, this function is not increasing as can be seen by considering that $0.45=f(0.2,0.45)>f(0.2,0.6)=0.2$ despite $0.45<0.6$.

Nevertheless, averaging functions are meant to provide a value that is representative of the information included in the inputs. Therefore, it seems logical to ask them to provide larger outcomes from larger incomes. In other words, monotonicity is a very intuitive requirement for averaging functions.

If we only require monotonicity and limit conditions, we obtain the definition of aggregation function that has been commonly accepted in the scientific community in the last years (see, for example,~\cite{beliakov2007aggregation}).

\begin{definition}\label{def:aggregfunction_number}
A function $f:\II^n\to \II$ is said to be an aggregation function if it holds that:
    \begin{itemize}
        \item $f$ is increasing.
        \item $f(\infII,\dots,\infII)=\infII$ and $f(\supII,\dots, \supII)=\supII$.
    \end{itemize}
\end{definition}

There are many functions that comply with this definition due to the few requirements imposed. Some examples are t-norms and t-conorms, usually defined on the interval $[0,1]$. They are aggregation functions, but they are not averaging functions in the sense of Cauchy's definition since they are not internal, except for the greatest t-norm, the minimum function, and the smallest t-conorm, the maximum function.

A modern approach to the concept of aggregation function is given by Grabisch through the ideas of idempotizable function and diagonal section.

\begin{definition}~\cite{grabisch2009}
    The diagonal section of any function $F:\II^n\to\RR$ is the unary function $\delta_F:\II\to\RR$ defined as $\delta_F (x)=F(x,\dots,x), \forall x\in\II$.
\end{definition}

\begin{definition}~\cite{grabisch2009}
    A function $F:\II^n\to\overline{\RR}$ is idempotizable if its diagonal section $\delta_F$ is strictly increasing and satisfies $ran(\delta_F)=ran(F)$, where $ran$ is the image of the function.
\end{definition}

This definition is closely related to the concept of idempotent function, which is central to the study of averaging functions. Let us first recall the definition of an idempotent function.

\begin{definition}
    A function $f:\II^n\to\II$ is said to be idempotent if $f(x, \dots,x)=x$ for all $x\in \II$.
\end{definition}

It is clear that if $f:\II^n\to \II$ is a monotone, idempotent function, its diagonal section is the identity function so $\II=ran(\delta_f)\subseteq ran(f)\subseteq \II$.

An example of an idempotizable function that is not an idempotent function is the function $g:[0,1]^2\to[0,1]$ given by $g(x,y)=xy$, that is, the product.

Once these notions have been recalled, we can introduce the definition of average proposed by Grabisch~\cite{grabisch2009}.

\begin{definition}[\textbf{Grabisch's definition of average}]   \label{def:average_numeros}  A function $f:\II^n\to\II$ is an average function in $\II^n$ according to Grabisch if there exists an increasing and idempotizable function $F:\II^n\to\RR$ such that $F=\delta_F\circ f$. In this case, we say that $f$ is an average associated with $F$ in $\II^n$.
\end{definition}

Although it may seem that Grabisch's and Cauchy's definitions are very different, the following result shows the connection between them.

\begin{proposition}\label{prop1}
    Let $f:\II^n\to\II$. The following statements are equivalent:
    \begin{enumerate}[i)]
        \item $f$ is idempotent and increasing.
        \item $f$ is internal and increasing.
        \item $f$ is an average function according to Grabisch's definition.
    \end{enumerate}
\end{proposition}

\begin{proof}
    Let $f:\II^n\to\II$ be a function and let $(x_1,\dots,x_n)$ be an element in $\II^n$.

Let us denote  $M=\maxi\{x_1,\dots,x_n\}$ and $m=\mini\{x_1,\dots,x_n\}$. If $f$ is idempotent and increasing, it follows that

    $$m=f(m,\dots,m)\leq f(x_1,\dots,x_n)\leq f(M,\dots,M)=M$$

    hence, the function is internal.

    Conversely, if we assume that the function $f$ is internal, for every $x\in\II$

    $$x=\mini\{x,\dots,x\}\leq f(x,\dots,x)\leq \maxi\{x,\dots,x\}=x$$

    this implies that $f$ is idempotent.

On the other hand, if $f:\II^n\to\II$ is idempotent and increasing, $\delta_f:\II\to\RR$ is the identity function because for every $x\in\II$, $\delta_f(x)=f(x,\dots,x)=x$. This implies that $f=\delta_f\circ f$, that is, $f$ is an average function according to Grabisch's definition.

    Finally, if $f$ is an average according to Grabisch's definition, then there exists an increasing and idempotizable function $F$ such that $F=\delta_F\circ f$ and as $F$ is idempotizable, $\delta_F$ is, by definition, strictly increasing and $ran(\delta_F)=ran(F)$, so it exists $\delta_F ^{-1}$ and this implies that we can write $f=\delta_F^{-1}\circ F$, then

    $$f(x,\dots,x)=(\delta_F^{-1}\circ F)(x,\dots,x)=\delta_F^{-1}(F(x,\dots,x))=x$$

    so $f$ is idempotent.
\end{proof}

At the sight of Proposition~\ref{prop1},  we have three equivalent ways to formalize the notion of average function according to Grabisch, and so, from now on, we will just call it average function.

Classical examples of average functions are:
\begin{itemize}
  \item The generalized $m$-power mean for $m \ge 1$:
  $$
  m\mbox{pm}(x_1, \dots, x_n)=\sqrt[m]{\frac{\sum_{i=1}^{n}x^m_i}{n}}\,.
  $$
  where $x_1,\dots,x_n$ are real values and taking only the positive root for even powers.

The case $m=1$ is the usual arithmetic mean denoted here as $mean$, but usually denoted as $\overline{x}$ in Statistics.

  \item The minimum and maximum operators, denoted by $\mini$ and $\maxi$. They always exist since the usual order considered in $\mathbb{R}$ is total.
  \item The median $Me$. Since we consider it here as a function in $\reales$, it should provide just one output. Therefore, we define it as the central value (once the inputs are ordered) if $n$ is odd and as the middle point of the two central values if $n$ is even. Thus,  $Me(1,8,3,5,4)=4$ and
      $Me(5,2,8,9,3,12)$ is  the middle point of $5$ and $8$, that is, $6.5$.
\end{itemize}
In the following section, we study their generalizations to the interval data context.


\section{On the definition of average for interval data}\label{sec:average_interval}

The main purpose of this section is the study of the definition of average function for interval data. We also present a proposal for the notion of aggregation function in this context.

For real numbers, according to Proposition~\ref{prop1}, there are three equivalent definitions of an average function. However, in the general context of intervals, even basic concepts such as order become non-trivial.  To avoid introducing unnecessary complexity into the study, we adopt the simplest of the aforementioned equivalent definitions as our starting point. Therefore, we will consider the definition based on idempotency and monotonicity. Consequently, it is necessary to establish definitions for idempotency and monotonicity in the general context of interval data.

The idea of idempotency admits the following generalization:
\begin{definition}
A function $h:\LR^n\to\LR$ is idempotent if for every $a \in \LR$ it holds that
$
h(a,\dots,a)=a\,.
$
\end{definition}

The concept of growth is contingent upon the existence of an order that enables the comparison of intervals. If the order changes, the notion of growth changes accordingly. As previously stated, there is no singular, universally accepted, and universally used order for intervals, as there is for real numbers. Consequently, the generalization of the definition of monotonicity is not a singular and fixed concept; rather, it is contingent upon the initial order that has been established.

\begin{definition}
   Let $\preceq$ be an order on $\LL$. A function $h:\LL^n\to \LL$ is said to be \mbox{$\preceq$-increasing} if
   for every $a_1,\dots,a_n,b\in \LL$ such that for some $i \in\{1,\dots,n\}$ it holds that $a_i \preceq b$, it satisfies that
$
h(a_1,\dots,a_i,\dots,a_n)\preceq h(a_1,\dots,b,\dots,a_n).
$
\end{definition}

Once the concept of monotonicity has been established in the context of interval data, it is a straightforward process to define an aggregation function over $\II \in \LL$.

\begin{definition}
Fix an element $\II \in \LL$.
   Let $\preceq$ be an order on $\LL$ such that there exist zero and one element for that order on $\Ldinter$. We will denote those two elements, the zero and one elements with respect to this order by $0_{\preceq}^{\II}$ and $1_{\preceq}^{\II}$, respectively. A function $h:\Ldinter^n\to \Ldinter$ is said to be an $\preceq$-aggregation function if it is $\preceq$-increasing and
 $h(0_{\preceq}^{\II},\dots,0_{\preceq}^{\II})=0_{\preceq}^{\II}$ and $h(1_{\preceq}^{\II},\dots,1_{\preceq}^{\II})=1_{\preceq}^{\II}.$

\end{definition}

Observe that the previous definition cannot be extended in general to a function $h$ defined on $\LL$ since for many orders  there is neither zero nor one element in $\LL$.

In the context of real numbers, one advantage of average functions over aggregation functions is that they do not require the existence of zero and one elements. This advantage is retained when moving into the context of intervals as we next show.

\begin{definition}\label{def:intervalaverage}
Let us consider an order $\preceq$ on $\LL$. A function $h:\LR^n \to\LR$ is an interval average function with respect to the order $\preceq$, denoted $\preceq$-average, if it is
 idempotent and
  $\preceq$-increasing.
\end{definition}

In the case of real numbers, Proposition~\ref{prop1} shows that two seemingly distinct definitions of an averaging function are, in fact, equivalent. To determine whether this equivalence also holds for intervals, we must first define what it means for a function to be internal in the context of intervals. This requires specifying the bounds that will constrain the function. If we adhere strictly to the definition used for real numbers, we need to formalize the notions of maximum and minimum elements within a set of intervals. To do so, we must be able to compare the intervals, which is only possible once an order has been established among them. Moreover, to guarantee the existence of a maximum and a minimum for any finite set of intervals, it is necessary that every pair of intervals be comparable; otherwise, these extremal elements may not exist. Therefore, not only is an order required, but it must also be a total order to ensure that both the maximum and minimum exist for every finite tuple of intervals.

\begin{definition}\label{def:minimo_funcion_intervalos}
Given a total order $\preceq$ on $\LL$, the operator $\mincq:\LL^n\to \LL$ is defined as
$$
\mincq\vecas=a_j,
$$
where $j\in\{1,\dots,n\}$ is such that $a_j\preceq a_i,\enskip\forall i\in\{1,\dots,n\}$.
\end{definition}

\begin{definition}\label{def:maximo_funcion_intervalos}
Given a total order $\preceq$ on $\LL$, the operator $\maxto:\LL^n\to \LL$ is defined as
$$
\maxto\vecas=a_j,
$$
where $j\in\{1,\dots,n\}$ is such that $a_j\succeq a_i\enskip\forall i\in\{1,\dots,n\}$.
\end{definition}

This is a logical definition from the ontic point of view. If the intervals are regarded as the smallest interesting piece of information and they can be arranged with respect to an order, the minimum with respect to that order is the first element of the resulting chain, and the maximum is the last element.

As previously explained, if the order is not total, it is not possible to ensure the existence of both the minimum and the maximum for every set of intervals. This is the case of the lattice order. For every pair of intervals that are incomparable according to this order, there is neither maximum nor minimum. Consider, for example, the pair $[2,5]$ and $[3,4]$. These two intervals are not comparable according to the lattice order. Therefore, we can neither set their maximum nor their minimum.

However, since the lattice order induces a lattice structure in $\LL$, we know that for every pair (and therefore for every finite set) of intervals, there exist an infimum and a supremum element. This holds for any order that induces a lattice structure.

\begin{definition}\label{def:infimum}
  Given an order $\preceq$ that induces a lattice structure on $\LL$, the infimum of a finite subset $\setas$ of $\LL$ is an interval $a_{inf}^{\preceq} \in \LL$ such that
  $a_{inf}^{\preceq} \preceq a_i$ for all $i \in \{1,\dots,n\}$ and
  for any other $b \in \LL$ such that $b \preceq a_i$ for all $i$, it also holds that $b \preceq a_{inf}^{\preceq}$.

  We will denote by $\infto$ the function that assigns the infimum element to every finite tuple. This is, the function $\infto: \LL^n\to \LL$ such that $Inf_{\preceq}\vecas=a_{inf}^{\preceq}\,.
$
\end{definition}

\begin{definition}\label{def:supremum}
  Given an order $\preceq$ that induces a lattice structure on $\LL$, the supremum of a finite subset $\setas$ of $\LL$ is an interval $a_{sup}^{\preceq} \in \LL$ such that
 $a_i\preceq a_{sup}^{\preceq}$ for all $i \in \{1,\dots,n\}$ and
    for any other $b \in \LL$ such that $a_i \preceq b $ for all $i$, it also holds that $a_{sup}^{\preceq} \preceq b$.

  We will denote by $\supto$ the function that assigns the supremum element to every finite tuple. This is, the function $\supto: \LL^n\to \LL$ such that $Sup_{\preceq}\vecas=a_{sup}^{\preceq}\,.
$
\end{definition}

For the lattice order, we have obtained explicit expressions for the infimum and the supremum. We will denote the infimum according to the lattice order by $\inflo$ and the supremum by $\suplo$.

\begin{proposition}\label{prop:inflo}
The operator $\inflo: \LL^n \to \LL$ can be expressed as
$$\inflo\vecas=[\mini\{\unda_1,\dots \unda_n\},\mini\{\ova_1,\dots,\ova_n\}],$$
where $\vecas=(\colec)$.
\end{proposition}

\begin{proof}
    Let $\{a_i=[\underline{a}_i,\overline{a}_i]\}_{i=1}^n \subseteq \LL$. Then, let us consider the interval $[\mini\{\unda_1,\dots \unda_n\},\mini\{\ova_1,\dots,\ova_n\}]$. Since $\mini\{\unda_1,\dots,\unda_n\}\leq \underline{a}_i\ and\ \mini\{\ova_1,\dots,\ova_n\}\leq \overline{a}_i,\ \forall i\in\{1,\dots,n\}$, it is clear that
    $$
    [\mini\{\unda_1,\dots,\unda_n\},\mini\{\ova_1,\dots,\ova_n\}]\lator[\unda_i,\ova_i]\,
    $$
    $\forall i\in\{1,\dots,n\}$. Now, let $a=[\unda,\ova]\in\LL$ such that $a\lator a_i, \forall i\in\{1,\dots,n\}$, this is, $\unda\leq \unda_i\ and\ \ova\leq \ova_i \;\forall i\in\{1,\dots,n\}$, therefore $\unda\leq \mini\{\unda_1,\dots,\unda_n\}$ and $\ova\leq \mini\{\ova_1,\dots,\ova_n\}$. It follows $$[\unda,\ova]\lator[\mini\{\unda_1,\dots,\unda_n\},\mini\{\ova_1,\dots,\ova_n\}].$$
\end{proof}

Note that the resulting interval is not ne\-ce\-ssa\-ri\-ly one of the intervals of the departing set. For example, observe that $\inflo([2,5],[3,4])=[2,4]$.

However, this definition fits better if we consider intervals from the epistemic point of view: if each interval represents the set of values that a decision maker considers that are possible, then the interval that represents the minimum of all the answers given by the experts should be the interval bounded between the minimum of the lower bounds and the minimum of the upper bounds.

Analogously, we can prove the following result.

\begin{proposition}\label{prop:suplo}
The operator $\suplo: \LL^n \to \LL$ can be expressed as
$$\suplo\vecas=[\maxi\{\unda_1,\dots \unda_n\},\maxi\{\ova_1,\dots,\ova_n\}],$$
where $\vecas=(\colec)$.
\end{proposition}

Neither in this case does the aggregation necessarily result in an interval from the set being aggregated as can be seen using the same example:
$
\suplo([2,5],[3,4])=[3,5]
$. Of course, this is not a necessary requirement, since even in the classical case of aggregating numbers this property is not fulfilled in general.

\begin{remark}
  For the particular case of total orders, the infimum and supremum of every finite set of elements becomes the minimum and the maximum respectively. Therefore, the infimum and supremum functions previously defined become the minimum and maximum functions given in Definitions~\ref{def:minimo_funcion_intervalos} and~\ref{def:maximo_funcion_intervalos}.
\end{remark}

%
%
%

Let us recall that we are trying to extend the notion of internal function to intervals and that, in order to do that, we need to set what we consider as the limits. We have seen that the notions of minimum and maximum used for numbers can be generalized to infimum and supremum in the context of intervals, and the existence of such bounds is guaranteed as long as the order induces a lattice structure.
We are now in a position to define internal function with respect to an order for intervals.

\begin{definition}\label{def:internal_function_intervals}
Given $\preceq$ an order that induces a lattice structure on $\LL$, the function $h:\LL^n\to \LL$ is internal with respect to the order $\preceq$ (shortly, $\preceq$-internal) if $Inf_{\preceq} (a_1, a_2, \ldots, a_n) \preceq h(a_1, a_2, \ldots, a_n) \preceq Sup_{\preceq}(a_1, a_2, \ldots, a_n)$, for any $a_1, a_2, \ldots, a_n\in \LL$.
\end{definition}

Observe that for total orders if a function $h$ is not internal is because either $h\vecas \prec \mini_{\preceq}\vecas
$ or
$
\maxi_{\preceq}\vecas \prec h\vecas
$,
for some tuple $\vecas$. However, when the order is not total it can also happen that $h\vecas$ is not comparable with the infimum, $\infto \vecas$ or with the supremum, $\supto \vecas$.
Consider the function $h:\LL^2\to \LL$ given by $h(a_1,a_2)=[0,1]$. If we consider the intervals $[0.1,0.3],[0.2,0.4]$ it holds that
$h([0.1,0.3],[0.2,0.4])=[0,1]$, $\inflo([0.1,0.3],[0.2,0.4])=[0.1,0.3]$ and $\suplo([0.1,0.3],[0.2,0.4])=[0.2,0.4]$. So $h([0.1,0.3],[0.2,0.4])$ is neither comparable (according to the lattice order) to $\inflo([0.1,0.3],[0.2,0.4])$ nor to $\suplo([0.1,0.3],[0.2,0.4])$.

Once introduced the notion of internal function for intervals, we can study the equivalence of Proposition~\ref{prop1} in the context of interval data. And we can provide a positive answer: the equivalence holds in this more general context too for the case of orders inducing a lattice structure.

\begin{proposition}
    Let $\preceq$ be an order that induces a lattice structure on $\LL$ and let \mbox{$h:\LL^n\to \LL$} be an $\preceq$-increasing function. Then, the following statements are equivalent:
    \begin{enumerate}[i)]
        \item $h$ is $\preceq$-internal.
        \item $h$ is idempotent.
    \end{enumerate}
\end{proposition}

\begin{proof}
Let us suppose that $\preceq$ is an order relation that induces a lattice structure on $\LL$.

Let us suppose that   $h$ is $\preceq$-internal. For any $a \in \LL$, we obtain that
$Inf_{\preceq}(a,\dots,a)\preceq h(a,\dots,a)\preceq Sup_{\preceq}(a,\dots,a)$. As    it is immediate that $Inf_{\preceq}(a,a\ldots,a)=a$ and
$Sup_{\preceq}(a,a\ldots,a)=a$, then
$a\preceq h(a,a,\ldots,a)\preceq a$. Applying now the antisymmetry of $\preceq$, we obtain that $h(a,a,\ldots,a) =a$, that is,
  $h$ is idempotent.

Conversely, let us consider that  $h$ is idempotent.  Then, for any $a_1,\dots,a_n\in\LL$ we have that
$
Inf_{\preceq}\vecas =h(Inf_{\preceq}\vecas, \dots, Inf_{\preceq}\vecas)
$
and
$
Sup_{\preceq}\vecas =h(Sup_{\preceq}\vecas, \dots, Sup_{\preceq}\vecas)
$.
Moreover, since $h$ is $\preceq$-increasing, we have that
$
h(Inf_{\preceq}\vecas,\dots,Inf_{\preceq}\vecas) \preceq h\vecas
$
and also that
$
h(Sup_{\preceq}\vecas, \dots, Sup_{\preceq}\vecas) \geq h\vecas
$.
     Hence, the function $h$ is internal.
\end{proof}

\section{On classical interval average operators}\label{sec:classical_interval_averages}

In this section, we examine several of the most classical aggregation operators in the context of real numbers. We recall their generalizations to interval data and investigate whether they can also be regarded as average operators in this broader setting. The concept of averaging inherently involves monotonicity, which itself depends on the chosen order. In this work, we focus on the lattice order and the orders $\preceq_{\alpha \beta}$ defined in Eq.~\ref{ordenesKalfa}. Nevertheless, as demonstrated in Proposition~\ref{prop:minymaxaverageparatodoorden}, a general result can be established for all admissible orders.

\subsection{Separable functions}

Some operators over intervals can be expressed using two functions defined over real values (see, for example,~\cite{DiazVazquez2024new}). Following the terminology commonly used in the literature, we refer to such functions as separable. We formally define them in Definition~\ref{def:separable}, and subsequently analyze some of their general properties.

\begin{definition}\label{def:separable}
    A function $h:\LL^n\to \LL$ is called a separable function if there exist $f$ and $g$ satisfying $f, g: \RR^n \to \RR$ with
  \begin{equation}\label{eq:hcomo2fg}
  h(\colec)=[f(\unda_1,\dots,\unda_n),g(\ova_1,\dots,\ova_n)]\,,
  \end{equation}
where f and g are called component functions of h. In particular, f is called the lower component function of h and g is called the upper component function of h.
\end{definition}

Not every interval aggregation function $h:\LL^n\to\LL$ is separable. The minimum with respect to the lexicographic order of type 1 (and also with respect to the lexicographic order of type 2) is an example of non separable interval aggregation function.

\begin{remark}
Let $(a_1,\dots,a_n) \in \LL^n$ such that $\unda_1 < \unda_i$ for all $i \ge 2$ and $\ova_1 \neq \ova_n$.

If $\minlexuno$ is a separable function, with $f$ and $g$ as the lower
and upper component functions, we obtain that
$\minlexuno(\colec) = [\unda_1, \ova_1]$. This implies that $f(\unda_1, \dots, \unda_n) = \unda_1$ and
$g(\ova_1, \dots, \ova_n) = \ova_1$.

Moreover,
$\minlexuno([\unda_n, \ova_1], [\unda_2, \ova_2], \dots, [\unda_{n-1}, \ova_{n-1}], [\unda_1, \ova_n]) = [\unda_1, \ova_n]$, which implies that  $g(\ova_1, \dots, \ova_n) = \ova_n$. This leads to a contradiction and therefore the function $\minlexuno:\LL^n\to\LL$ is not separable.
\end{remark}

When a function $h:\LL^n\to \LL$ is separable, the study of some of the properties of the departure function $h$ can be transferred to the study of the corresponding property in its components. In the following results we explore some of those properties.

We begin by establishing a characterization of separable functions.

\begin{proposition}\label{prop:idempotency_h<->fg}
  Let $h:\LL^n\to \LL$ be a separable function with component functions f and g, then $h$ is idempotent if and only if both $f$ and $g$ are idempotent.
\end{proposition}

\begin{proof}
  If $h$ is idempotent then for any $x \in \reales$ it holds that $[x,x] \in \LL$ and
  $
 [x,x]=  h([x,x],\dots,[x,x])= [f(x,\dots,x), g(x,\dots,x)]
  $, and so $f(x,\dots,x)=x$ and $g(x,\dots,x)=x$.

  Conversely, if  $f$ and $g$ are idempotent in $\reales$, then for any $[\unda,\ova] \in \LL$ it holds that
  $h([\unda,\ova],\dots,[\unda,\ova]) =  [f(\unda,\dots,\unda),g(\ova,\dots,\ova)] =
  [\unda,\ova]$.
\end{proof}

We next examine the monotonicity of separable functions in relation to their components.

\begin{proposition}\label{prop:monotonicity_lator}
  Let $h:\LL^n\to \LL$ be a separable function with lower and upper component functions f and g, respectively. Then, $h$ is $\lator$-increasing if and only if both $f$ and $g$ are increasing.
\end{proposition}

\begin{proof}
Assume $[\unda_1,\ova_1] \lator [\underline{b}_1,\overline{b}_1]$. This is equivalent to $\unda_1 \le \undb_1$ and $\ova_1 \le \ovb_1$.\\
If $f$ and $g$ are increasing, it holds that $f(\unda_1,\unda_2,\dots,\unda_n) \le f(\underline{b}_1,\unda_2,\dots,\unda_n)$ and $g(\ova_1,\ova_2,\dots,\ova_n) \le g(\overline{b}_1,\ova_2,\dots,\ova_n)$
and then $h(\colec) \lator h([\underline{b}_1,\overline{b}_1],[\underline{a}_2,\overline{a}_2],\dots,[\underline{a}_n,\overline{a}_n])
$.

To prove the converse implication, take $x_1 \le y_1$ and $x_2,\dots,x_n$ real values, then $[x_1,x_1] \lator [y_1,y_1]$. By hypothesis,
$
h([x_1,x_1],[x_2,x_2],\dots,[x_n,x_n]) \lator h([y_1,y_1],[x_2,x_2],\dots,[x_n,x_n])
$.
Then,
$
[f(x_1,x_2,\dots,x_n),g(x_1,x_2,\dots,x_n] \lator
 [f(y_1,x_2,\dots,x_n),g(y_1,x_2,\dots,x_n)]
$
and so
$
f(x_1,x_2,\dots,x_n) \le f(y_1,x_2,\dots,x_n)
$
and
$
g(x_1,x_2,\dots,x_n) \le g(y_1,x_2,\dots,x_n).
$

The case for the first component has been considered; the remaining cases can be established in an analogous manner.
\end{proof}

We now investigate how the monotonicity properties of the component functions influence the lexicographical ordering associated with a separable function.

\begin{proposition}\label{prop:monotonicity_lexunor}
  Let $h:\LL^n\to \LL$ be a separable function with lower and upper component functions $f$ and $g$, respectively.
Then,
\begin{enumerate}[i)]
  \item if $f$ is strictly increasing and $g$ is increasing, then $h$ is $\lexunor$-increasing,
  \item if $g$ is strictly increasing and $f$ is increasing, then $h$ is $\lexdosor$-increasing.
\end{enumerate}
\end{proposition}

\begin{proof}
Let us consider the case of  the lexicographic order of type 1. Assume $[\unda_1,\ova_1] \lexunor [\undb,\ovb]$. This is equivalent to $\unda_1 < \undb$ or $(\unda_1 = \undb$ and $\ova_1 \le \overline{b})$. Assume $\unda_1 < \undb$. Since $f$ is strictly increasing, it holds that
$
f(\unda_1,\unda_2,\dots,\unda_n) < f(\undb,\unda_2,\dots,\unda_n)
$
and then
$
h(\colec) \lexunor h([\undb,\ovb],[\unda_2,\ova_2],\dots,[\unda_n,\ova_n])
$.

If $\unda_1 = \undb$ and $\ova_1 \le \ovb$, then, $f(\unda_1,\unda_2,\dots,\unda_n) = f(\undb,\unda_2,\dots,\unda_n)$  and, since $g$ is increasing,
$g(\ova_1,\ova_2,\dots,\ova_n) \le g(\ovb,\ova_2,\dots,\ova_n)$. Thus,
we have that
$
h(\colec) \lexunor h([\undb,\ovb],[\unda_2,\ova_2],\dots,[\unda_n,\ova_n])
$
as we wanted to prove.



The case corresponding to the lexicographic order of type 2 is completely analogous and therefore the proof is omitted here.
\end{proof}

The previous conditions are not necessary. It can happen that $h$ is $\lexunor$-increasing, even though neither
$f$
 nor
$g$
 is strictly increasing.

\begin{example}
Consider the operators $f(x,y) =0$ and $g(x,y)=0$.
The separable function $h:\LL^2\to \LL$  with lower and upper component functions $f$ and $g$
is $\lexunor$ and $\lexdosor$-increasing despite neither $f$ nor $g$ are strictly increasing.
\end{example}

Now that we have seen that none of the previous implications hold in the reverse direction, we proceed to show that the imposed conditions cannot be weakened. That is, if the monotonicity is not strict or if one of the two components is not increasing, the function
$h$
 may fail to be increasing with respect to the corresponding orders.

\begin{example}
Consider the operators $f(x,y) =0$ and $g(x,y)=0$ if $x\le 0$, $g(x,y)=x$, otherwise. It is immediate that both are increasing. However, the separable function \mbox{$h:\LL^2\to \LL$}  with lower and upper component functions $f$ and $g$
is not $\lexunor$-increasing, since $h([2,7],[6,12]) = [f(2,6),g(7,12)] = [0,7]$
and $h([4,5],[6,12]) = [f(4,6),g(5,12)] = [0,5]$. Thus, $[2,7]\preceq_{Lex1} [4,5]$, but
$h([4,5],[6,12]) \prec_{Lex1} h([2,7],[6,12]) $.

If we consider the separable function $h'$ with lower component function $g'$ defined as $g'(x,y)=x$ if $x \le 0$, $g'(x,y)=0$ otherwise, and upper component function $f$, we obtain that $ h([-4,7],[6,12])  = [-4,0]  \prec_{Lex2} [-2,0 ]=h([-2,5],[6,12])$ even in the case \mbox{$[-2,5]\preceq_{Lex2} [-4,7]$}.

On the other hand, if we consider $f(x,y) =-e^{-x}$ and $g(x,y)=e^{-x}$, it holds that $f$ is strictly increasing, $g$ is strictly decreasing and $f(x,y) \le g(x,y)$ for all $(x,y) \in \reales^2$. Then $[0,1]\prec_{Lex1} [0,2]$ and
$h([0,2],[0,2]) = [f(0,0),g(2,2)] = [-1,\frac{1}{e^2}] \prec_{Lex1} [-1,\frac{1}{e}] = h([0,1],[0,2])$.
Thus, $h$ is not increasing, even when $f$ is strictly increasing, which shows that the monotonicity condition on $g$ is indeed necessary.

Similarly, if we consider $f(x,y) =-e^x$ and $g(x,y)=e^{x}$, then $f$ is strictly decreasing and $g$ is strictly increasing with $f(x,y) \le g(x,y)$ for all $(x,y) \in\reales^2$. Now, $[-2,0]\preceq_{Lex2} [-1,0]$ but
$h([-1,0],[-2,0]) = [f(-1,-2),g(0,0)] = [-\frac{1}{e},1] \prec_{Lex2} [-\frac{1}{e^2},1] = [f(-2,-2),g(0,0)] =
h([-2,0],[-2,0])$ and this example shows that $h$ may fail to be increasing, even when $g$ is strictly increasing, due to the fact that $f$ is not increasing.
\end{example}

The implications of Proposition \ref{prop:monotonicity_lexunor} do not hold in general for every
$\preceq_{\alpha,\beta}$-admissible order.

\begin{example}
Let us define the following functions, both strictly increasing in each variable:
$
f(x, y) = x + e^{y}$
 and $g(x, y) = e^{x} + y$.

 Then we have that
$h([1,2],[0,0]) = [2,e^2] \prec_{XY} [0,e^3]= h([-1,3],[0,0]) $,
despite $[-1,3]\preceq_{XY} [1,2]$.

 Thus, the separable function $h$  with lower and upper component functions $f$ and $g$  is not $\preceq_{XY}$-increasing, even in the case $f$ and $g$ are strictly increasing.
\end{example}

\subsection{Infimum and supremum}

Two of the most well-known averages for real numbers are the minimum and the maximum. In the context of intervals, only total orders allow to assure the existence of minimum and maximum.   However, orders as the lattice order, which induces a lattice structure on $\LL$, allow to guarantee the existence of infimum and supremum.  Thus, this type of orders can provide  a more general result, in the sense that the definition of average considered is not restricted to the lattice order or to the
$\preceq_{\alpha,\beta}$-orders.

\begin{proposition}\label{prop:minymaxaverageparatodoorden}
Given an order $\preceq$ that induces a lattice structure on $\LL$,
\begin{enumerate}[i)]
  \item the infimum function with respect to  $\preceq$, $\infto$,  is an $\preceq$-average.
  \item the supremum function with respect to  $\preceq$, $\supto$,  is an $\preceq$-average.
\end{enumerate}
\end{proposition}

\begin{proof}

\begin{itemize}
  \item  Let $\preceq$ be an order that induces a lattice structure on $\LL$.

  Firstly, it is clear that $\infto$ is an idempotent operator, since the infimum of a collection of identical elements is necessarily that element.

  Now, we are going to prove that $\infto$ is an $\preceq$-increasing operator. For this let $(\colec)\in\LL^n$ be a collection of intervals and let $[\undb,\ovb]$ be an interval such that $[\unda_i,\ova_i]\preceq[\undb,\ovb]$ for some $i\in\{1,\dots,n\}$. Without loss of generality, we are going to assume that $[\unda_1,\ova_1]\preceq[\undb,\ovb]$.

%
%
%

Thus, we have that $\infto(\colec)\preceq[\unda_i,\ova_i]\ \forall i\in\{1,\dots,n\}$ and so, by the transitivity of the order, we also have that $\infto(\colec)\preceq[\undb,\ovb]$.

Then, $\infto(\colec)\preceq\infto([\undb,\ovb],[\unda_2,\ova_2],\dots,[\unda_n,\ova_n])$
and this proves that the operator is increasing.

  \item The proof for $\supto$ is analogous to the previous one.
  \end{itemize}
\end{proof}

%
%
%
%

Since, for total orders, the minimum and the infimum coincide, it follows from the previous result that $Min_{\preceq_{ao}}$
  is a $\preceq_{ao}$-average for any total order $\preceq_{ao}$, and the same holds for the maximum. It also follows immediately that the operators
$\inflo$  and $\suplo$ are $\lator$-averages. However, it is important to emphasize that the order used to define the infimum or supremum is the same as the one with respect to which they act as average functions. These results do not generally hold under a different order, as illustrated in the following example.


\begin{example}
    The operator $\inflo: \LL^n\to \LL$ is not an $\preceq_{\alpha \beta}$-average for any $\alpha$, $\beta$ with $\alpha \neq \beta$, since it is not  $\preceq_{\alpha \beta}$-monotone.

Let us first set  $b, d$ two real values such that $\frac{1}{2}<d<b<d+\lambda_\alpha$ with
\[\lambda_\alpha=\left\{
\begin{array}{cl}
\frac{1-\alpha}{2\alpha} & \mbox{ if } \alpha\in (0,1)\\
\infty & \mbox{ otherwise}
\end{array}
\right.\]

    \begin{itemize}
    \item If $\alpha \neq 1$, it holds that $[0,b] \preceq_{\alpha \beta} [\frac{1}{2},d]$, since $\alpha b<(1-\alpha)\frac{1}{2}+\alpha d$, by considering that $b<d+\frac{1-\alpha}{2\alpha}$ if $\alpha >0$ and $0<(1-\alpha)\frac{1}{2}$ if $\alpha=0$.

    However,
   $\inflo([0,b+d],[\frac{1}{2},d]) = [0,d]  \prec_{\alpha \beta} [0,b] =
    \inflo([0,b+d],[0,b])$ since
$\alpha d < \alpha b$ if $\alpha>0$ and $\beta d < \beta b$ if $\alpha=0$, because this implies that $\beta>0$.


\item If $\alpha=1$, then  $[\frac{1}{2},d] \preceq_{1\beta}[0,b]$, since $d <  b$, but
$\inflo(\left[\frac{1}{4}, \frac{1+2d}{4}\right],[0,b])   \prec_{1 \beta}
    \inflo(\left[\frac{1}{4}, \frac{1+2d}{4}\right],[\frac{1}{2},d])$, by applying that
$\inflo(\left[\frac{1}{4}, \frac{1+2d}{4}\right],[\frac{1}{2},d]) =     \left[\frac{1}{4},\frac{1+2d}{4}\right] $ since
$\frac{1+2d}{4} < d \Leftrightarrow 1+2d < 4d  \Leftrightarrow 1< 2d  \Leftrightarrow 1/2<d$,
    $\inflo(\left[\frac{1}{4}, \frac{1+2d}{4}\right],[0,b]) = \left[0,\frac{1+2d}{4}\right]  $
since   $d<b$ and
$\left[0,\frac{1+2d}{4}\right]  \prec_{1 \beta}   \left[\frac{1}{4},\frac{1+2d}{4}\right] $.

         \end{itemize}

    \end{example}

In particular, this result indicates that the operator $\inflo$ is neither monotone with respect to any of the lexicographic orders nor with respect to the Xu and Yager order. The operator $\suplo$ exhibits the same behaviour.

\begin{example}
Let us consider $\alpha, \beta\in[0,1]$ with $\alpha \neq \beta$.
        \begin{itemize}

            \item If $\alpha\in(0,1)$, we consider a number $c<1/2$ and a number $a\in \left(c,\min\left\{\frac{1}{2},c+\frac{\alpha}{2(1-\alpha)}\right\}\right)$.

            It holds that $[a,\frac{1}{2}]\preceq_{\alpha,\beta}[c,1]$, since $(1-\alpha)a<(1-\alpha)c+\frac{\alpha}{2}$.
            But
            $\suplo([0,1],[c,1])=[c,1] \prec_{\alpha,\beta} [a,1] =
            \suplo([0,1],[a,\frac{1}{2}])$

            \item For $\alpha=0$, it holds that $[0,3] \preceq_{0 \beta} [1,2]$, but
         $\suplo([1,2],[1,2]) = [1,2] \prec_{0 \beta} [1,3]=
         \suplo([1,2],[0,3])
         $.

         \item For $\alpha=1$, it holds that $[1,1] \preceq_{1 \beta} [0,2]$, but
         $\suplo([0,2],[0,2]) = [0,2]\prec_{1\beta} [1,2] =
         \suplo([0,2],[1,1])
         $.
        \end{itemize}

    \end{example}

\subsection{The generalized power mean}

In this subsection we focus on the study of generalized $m$th power means for intervals. The first step is to define the powers of an interval.

\begin{definition}\label{def:secondpower}
    Given $[\unda,\ova]\in\LL$ and $m\in\naturales$ we define the $m$th power of the interval $[\unda,\ova]$ as

    $$[\unda,\ova]^m=\left\{\begin{array}{ll}
    [0,\maxi\{\unda^m,\ova^m\}]& \mbox{ if } 0\in [\unda,\ova]  \mbox{ and } m \mbox{ is an even number},\\[.2cm]
    \left[\mini\{\unda^m,\ova^m\},\maxi\{\unda^m,\ova^m\}\right] & \mbox{ otherwise.}
    \end{array}\right.
$$
\end{definition}

Because the power function increases when the exponent is odd, we have
$[\unda,\ova]^m = [\unda^m,\ova^m]$ in that case.

\begin{proposition}\label{prop:power}
    Given $[\unda,\ova]\in\LL$ and $m\in\naturales$ , we have that $[\unda_i,\ova_i]^m=[\mini_{x \in [\unda_i,\ova_i]}x^m,\maxi\{\unda_i^m,\ova_i^m\}]$.\end{proposition}

\begin{proof}
For $m$ odd, as $x^m$ is an increasing function, we have that $[\unda_i,\ova_i]^m=[\unda_i^m,\ova_i^m]$. Since $f(x)=x^m$ (for $m$ odd) is increasing in $\reales$, the proof follows.

For $m$ even, the upper bound of $[\unda_i,\ova_i]^m$ is, by definition, $\maxi\{\unda_i^m,\ova_i^m\}$, so there is nothing to prove concerning the upper bound. Concerning the lower bound, recall that the function $f(x)=x^m$ is decreasing for $x<0$, increasing for $x>0$ and reaches a minimum for $x=0$. We distinguish two cases:
\begin{itemize}
    \item If $0 \in [\unda_i,\ova_i]$, by definition $[\unda_i,\ova_i]^m=[0,\maxi\{\unda_i^m,\ova_i^m\}]$. The proof follows from the fact that $f(x)=x^m$ reaches its minimum in $[\unda_i,\ova_i]$ for $x=0$. Therefore, $0=\mini_{x \in [\unda_i,\ova_i]}x^m$.
    \item If $0 \notin [\unda_i,\ova_i]$, by definition, $[\unda_i,\ova_i]^m=[\mini\{\unda_i^m,\ova_i^m\},\maxi\{\unda_i^m,\ova_i^m\}]$ the function $f(x)=x^m$ is monotone in $[\unda_i,\ova_i]$ and therefore, $\mini_{x \in [\unda_i,\ova_i]}x^m=\mini\{\unda_i^m,\ova_i^m\}$.
\end{itemize}

 \end{proof}

\begin{remark}
    Notice that for any natural $m$ with $m \in \naturales$, the $m$th power of an interval is a subset of the interval obtained from multiplying the elements of the original interval $m$ times, that is, $[\unda,\ova]^m=\{a^m \mid a\in[\unda,\ova]\}\subseteq\{a_1\cdot\, \dots \,\cdot a_m \mid a_1,\dots, a_m\in[\unda,\ova]\}$.

Moreover,  we get the equality for intervals that only contain positive (negative) values or $m=1$. However, the content can be strict for intervals including both positive and negative values and $m\ge 2$. For example,
  $
  [-1,2]^2=[0,4] \subset [-2,4]=\{a_1 \cdot a_2 \mid a_1, a_2 \in [-1,2]\}
  $.

\end{remark}

In a similar manner, we can also define the root of an interval as follows. Note that for even indices, the definition is only valid for intervals with non-negative values. Furthermore, we consider only the principal (positive) root. However, if the index of the root is odd, the root can be defined without imposing restrictions on the lower bound of the interval.

\begin{definition}
    Let $[\unda,\ova]\in\LL$ and  $m \in \naturales$.We define the $m$th root of $[\unda,\ova]$ as
    $$\sqrt[m]{[\unda,\ova]}=\left[\sqrt[m]{\unda},\sqrt[m]{\ova}\right]$$
where $\underline{a}$ has to be greater than or equal to $0$ if $m$ is even.
\end{definition}

Having established the definitions of the $m$th power and $m$th root for intervals, we are now prepared to define the power mean of order $m$ for interval-valued data.

\begin{definition}
Let $m$ be a natural number. The function $mPM:\LL^n\to\LL$ given by $$mPM(\colec)=\sqrt[m]{\frac{[\unda_1,\ova_1]^m+\dots+[\unda_n,\ova_n]^m}{n}}$$
is called generalized power mean of order $m$.
\end{definition}

This definition can be characterized through the next result.

\begin{corollary}\label{cor:mPM}
    Let $\colec\in\LL$ be a collection of intervals and $m \in \naturales$. Its generalized mean of order $m$ can be written as
   $$
mPM(\colec)=
\left[\sqrt[m]{\frac{\sum_{i=1}^n\displaystyle\mini_{x_i \in [\unda_i,\ova_i]}x_i^m}{n}},\sqrt[m]{\frac{\sum_{i=1}^n \displaystyle\maxi\{\unda_i^m,\ova_i^m\}}{n}}\right]\,.
$$
\end{corollary}

\begin{proof}

The proof follows from Proposition \ref{prop:power} and  the definition of sum, product by a number and $m$th root of an interval. Observe that for $m$ odd the root is well defined since the radicand is an interval with lower bound greater than or equal to zero.
 \end{proof}

It is important to note that this characterization cannot be simplified for even powers. However, for odd powers, a simpler and more intuitive expression can be derived: when $m$ is an odd number, the lower and upper bounds of the generalized mean of order $m$ are the generalized means of order $m$ of the lower and upper bounds of the intervals being aggregated, respectively.
\begin{corollary}
    Let $m \in \naturales$ be an odd number. The generalized mean of order $m$ can be written as
   $$
   \begin{array}{rl}
mPM(\colec)   & =
\left[\sqrt[m]{\frac{\sum_{i=1}^n\unda_i^m}{n}},\sqrt[m]{\frac{\sum_{i=1}^n \ova_i^m}{n}}\right] \\[.4cm]
 & =\left[m\mbox{pm}(\unda_1,\dots,\unda_n), m\mbox{pm}(\ova_1,\dots,\ova_n)\right]\,.
 \end{array}
$$
\end{corollary}

This definition is suitable if we consider intervals from an epistemic point of view, since it is the interval that contains all the possible generalized means of order $m$ of the possible outcomes.\\
Given the way in which we have defined arithmetics between intervals, it is also a suitable definition from an ontic point of view. It generalizes the definition for real numbers to intervals in a straightforward way.

Therefore, in the case where $m$ is odd, the generalized mean is a separable function whose lower and upper bounds are given by the generalized $m$-power means for real numbers ($m$pm). Idempotency follows directly from this property.

\begin{corollary}\label{cor:mPMidempotent}
 Let $m \in \naturales$ be an odd number. The operator $mPM:\LL^n\to\LL$ is idempotent.
\end{corollary}
\begin{proof}
  It follows from Proposition~\ref{prop:idempotency_h<->fg}, given that $mpm$ is idempotent in $\reales$.
\end{proof}

Thus, for $m$ an odd number, $mPM$ is an $\preceq$-average if and only if it is $\preceq$-increasing.

\begin{proposition}  \label{prop:su2} Let $m \in \naturales$ be an odd number.
    The operator $mPM:\LL^n\to\LL$ is a
$\lator$-average, a   $\lexunor$-average and a
  $\lexdosor$-average.
\end{proposition}
\begin{proof}
    Since $mpm$ is strictly increasing, it follows from Propositions~\ref{prop:monotonicity_lator} and~\ref{prop:monotonicity_lexunor} that $mPM$ is $\lator$-increasing, $\lexunor$-increasing and  $\lexdosor$-increasing.
    Since it is also idempotent by Corollary~\ref{cor:mPMidempotent}, the proof follows.
\end{proof}

The previous proposition does not hold if $m$ is an even number. Neither idempotency nor monotony can be assured as the two following counterexamples show.

\begin{example}
    The operator $mPM:\LL^n\to\LL$ is not idempotent for any $m \in \naturales$ if $m$ is even. To  prove that, we can consider the interval $[-1,1]$. For any even number $m$, it holds that $[-1,1]^m=[0,1]$, and so
$$
mPM([-1,1],\stackrel{n}{\dots},[-1,1])=\sqrt[m]{\frac{[0,1]+\stackrel{n}{\dots}+[0,1]}{n}}=\sqrt[m]{[0,1]}=[0,1]\neq [-1,1].
$$
\end{example}

\begin{example} \label{exem:su3}
    The operator $mPM:\LL^n\to\LL$ is not $\lator$-increasing for any $m \in \naturales$ if $m$ is even. In this case, we can consider the intervals $[0.5,1]$, $[-1,0.1]$ and $[0,0.75]$. It is clear that $[-1,0.1] \lator [0,0.75]$, but  $m$PM$([0.5,1],[-1,0.1])=\left[\sqrt[m]{\frac{0.5^m}{2}},1\right]$ whereas   $m$PM$([0.5,1],[0,0.75])=\left[\sqrt[m]{\frac{0.5^m}{2}},\sqrt[m]{\frac{1+0.75^m}{2}}\right]$. Thus,  for every $m\geq 2$ it holds that $\left[\sqrt[m]{\frac{0.5^m}{2}},\sqrt[m]{\frac{1+0.75^m}{2}}\right]\prec_{Lo}\left[\sqrt[m]{\frac{0.5^m}{2}},1\right]$ given that $\sqrt[m]{\frac{1+0.75^m}{2}}<1$ for every $m \ge2$. Then, the operator $m$PM is not $\lator$-increasing if $m$ is even.
\end{example}

As a consequence, for even indexes $mPM$ is not $\ador$-increasing for any admissible order $\ador$.

\begin{proposition}
The operator $mPM:\LL^n\to\LL$ is:
\begin{itemize}
\item a
$\lator$-average, a   $\lexunor$-average and a
  $\lexdosor$-average if $m$ is odd and $m\ge 3$,
  \item  neither a $\lator$-average nor a $\ador$-average for any admissible order if $m$ is an even number,
  \item a
$\lator$-average and an $\preceq_{\alpha\beta}$-average for every $\alpha,\beta \in [0,1]$ with $\alpha\neq \beta$ if $m=1$.
\end{itemize}

\end{proposition}

\begin{proof}
\begin{itemize}
\item This was proven in Proposition \ref{prop:su2}.
\item It is an immediate consequence of Example \ref{exem:su3}.

\item The property for the lattice order was proven in Proposition \ref{prop:su2}.

For the generated admissible orders,
we know from Corollary~\ref{cor:mPMidempotent} that
 $1$PM is an idempotent function. Then, it suffices to prove that $1$PM is $\preceq_{\alpha \beta}$-increasing.\\
Let $\alpha\in (0,1)$ (the cases $\alpha=0$ and $\alpha=1$ have been proven in Proposition~\ref{prop:monotonicity_lexunor}) and let $\colec,[\undb,\ovb]\in\LL$. Assume $[\unda_i,\ova_i]\preceq_{\alpha \beta}[\undb,\ovb]$ for some $i\in\{1,\dots,n\}$. Since $1$PM is a symmetric operator, we can assume without loss of generality that $i=1$. Hence it holds that $K_{\alpha}([\unda_1,\ova_1])<K_{\alpha}([\undb,\ovb])$ or $K_{\alpha}([\unda_1,\ova_1])=K_{\alpha}([\undb,\ovb])$ and $K_{\beta}([\unda_1,\ova_1])<K_{\beta}([\undb,\ovb])$ for every $\beta \neq \alpha$. Since we can rewrite $1\mbox{PM}(\colec)$ as
$$
\left[\frac{\unda_1}{n}+\frac{\unda_2+\dots+\unda_n}{n},
\frac{\ova_1}{n}+\frac{\ova_2+\dots+\ova_n}{n}\right]
$$
if we evaluate the operator $K_{\alpha}$ on the mean of the intervals $\colec$, we get
$$
(1-\alpha)\frac{\unda_1}{n}+\alpha \frac{\ova_1}{n}+(1-\alpha)\left(\frac{\unda_2+\dots+\unda_n}{n}\right)+ \alpha\left(\frac{\ova_2+\dots+\ova_n}{n}\right)
$$

Similarly, if we evaluate the operator $K_{\beta}$ we get
\[
(1-\beta)\frac{\unda_1}{n}+\beta \frac{\ova_1}{n}+(1-\beta)\left(\frac{\unda_2+\dots+\unda_n}{n}\right)+ \beta\left(\frac{\ova_2+\dots+\ova_n}{n}\right)
\]
Now, since $[\unda_1,\ova_1] \preceq_{\alpha \beta} [\undb,\ovb]$ it holds that $(1-\alpha)\unda_1 + \alpha \ova_1<(1-\alpha)\undb+\alpha \ovb$ or $(1-\alpha)\unda_1+\alpha \ova_1=(1-\alpha)\undb + \alpha \ovb$ and $(1-\beta)\unda_1+\beta \ova_1\leq (1-\beta)\undb + \beta \ovb$. If we replace $[\unda_1,\ova_1]$ by $[\undb,\ovb]$, it follows clearly that\\
$(1-\alpha)\frac{\unda_1}{n} + \alpha\frac{\ova_1}{n} + (1-\alpha)\left(\frac{\unda_2+\dots+\unda_n}{n}\right)+ \alpha\left(\frac{\ova_2+\dots+\ova_n}{n}\right)<(1-\alpha)\frac{\undb}{n}+\alpha\frac{\ovb}{n}+(1-\alpha)\left(\frac{\unda_2+\dots+\unda_n}{n}\right)+\alpha \left(\frac{\ova_2+\dots+\ova_n}{n}\right)$\\
or\\
$
(1-\alpha)\frac{\unda_1}{n} + (1-\alpha)\left(\frac{\unda_2+\dots+\unda_n}{n}\right)+ \alpha\frac{\ova_1}{n} +\alpha\left(\frac{\ova_2+\dots+\ova_n}{n}\right)=
(1-\alpha)\frac{\undb}{n}+(1-\alpha)\left(\frac{\unda_2+\dots+\unda_n}{n}\right)+\alpha\frac{\ovb}{n}+\alpha \left(\frac{\ova_2+\dots+\ova_n}{n}\right)
$\\
and\\
$
(1-\beta)\frac{\unda_1}{n} + (1-\beta)\left(\frac{\unda_2+\dots+\unda_n}{n}\right)+ \beta\frac{\ova_1}{n} +\beta\left(\frac{\ova_2+\dots+\ova_n}{n}\right)\le
(1-\beta)\frac{\undb}{n}+(1-\beta)\left(\frac{\unda_2+\dots+\unda_n}{n}\right)+\beta\frac{\ovb}{n}+\beta \left(\frac{\ova_2+\dots+\ova_n}{n}\right)
$

Thus, $
1PM([\unda_1,\ova_1],[\unda_2,\ova_2],\dots,[\unda_n,\ova_n])\preceq_{\alpha \beta}
1PM([\undb,\ovb],[\unda_2,\ova_2],\dots,[\unda_n,\ova_n])
$.

\end{itemize}
\end{proof}

Furthermore, the lexicographic orders of types 1 and 2 are the only instances of the $\preceq_{\alpha \beta}$-orders defined in~\cite{Bustince} for which it holds true if $m\neq 1$.

\begin{example} We will provide an example to show that the operator $m$PM$:\LL^2\to \LL$, with $m \geq 2$, is not $\preceq_{\alpha \beta}$-increasing for any $\alpha \in (0,1)$ and $\beta \neq \alpha$. Thus, we will prove that there exist $[a_2,b_2], [a_2', b_2']\in \LL$ such that
 $
 \terv{a}{b}{2} \preceq_{\alpha \beta} \terv{a'}{b'}{2}
 $,
but
$m\textrm{PM}([0,b],\terv{a'}{b'}{2}) \prec_{\alpha \beta}
m\textrm{PM}([0,b],\terv{a}{b}{2})$.

First, consider the function
$f(x)=\sqrt[m]{k+x^m}-x$ defined in $[0,\infty)$ with $k>0$.
The derivative of $f$ is
$$
f'(x)=\frac{1}{m}(k+x^m)^{\frac{1-m}{m}}(m x^{m-1})-1\,,
$$
then $f'(x)<0$ is equivalent to
$$
\frac{1}{m}(k+x^m)^{\frac{1-m}{m}}(m x^{m-1})<1.
$$
This is equivalent to
$$
\frac{x^{m-1}}{\sqrt[m]{(k+x^m)^{m-1}}}<1 \Leftrightarrow x^{m-1} <\sqrt[m]{(k+x^m)^{m-1}}
$$
raising to the  $m$th-power,
$$
(x^m)^{m-1} < (k+x^m)^{m-1} \qquad \Leftrightarrow \qquad x^m < k+x^m
$$
which holds for any $x \ge 0$ since $k>0$. Therefore, the function $f$ is strictly decreasing in $[0,\infty)$.

Secondly,  let us consider $b>0$ and $b_2>1$. As $\alpha < 1$. then $1-\alpha>0$ and therefore there exists $b'_2$ with $b_2<b_2'<b_2+(1-\alpha)$. If we consider $k=b^m$, we have that $f(b_2')<f(b_2) $, that is, $\sqrt[m]{k+{b'}_2^m} - {b'}_2 < \sqrt[m]{k + b_2^m} - b_2$, which is equivalent to $\left(\sqrt[m]{b^m+{b'}_2^m} - \sqrt[m]{b^m + b_2^m}\right) < {b'}_2-b_2$.

From here, for any $a_2'$, we have that
$${a'}_2 + \frac{\al}{1-\al}\left[\sqrt[m]{b^m+{b'}_2^m} - \sqrt[m]{b^m + b_2^m}\right] < {a'}_2 + \frac{\al}{1-\al}({b'}_2-b_2)$$

Thus, let $a_2$ be a number such that
$${a'}_2 + \frac{\al}{1-\al}\left(\sqrt[m]{b^m+{b'}_2^m} - \sqrt[m]{b^m + b_2^m}\right) < a_2< {a'}_2 + \frac{\al}{1-\al}({b'}_2-b_2).$$

From $a_2< {a'}_2 + \frac{\al}{1-\al}({b'}_2-b_2)$ we obtain that
$(1-\alpha)a_2+\alpha b_2 < (1-\alpha) {a'}_2+\alpha {b'}_2$. If we consider ${a'}_2={b'}_2-1$, we  have that ${a'}_2>0$ and  that $$a_2 < ({b'}_2-1) + \frac{\al}{1-\al}({b'}_2-b_2) =\frac{ {b'}_2-\alpha b_2}{1-\alpha}-1< \frac{[b_2+(1-\alpha)]-\alpha b_2}{1-\alpha}-1 = b_2.
$$
As $a_2<b_2$ and ${a'}_2<{b'}_2$, from $(1-\alpha)a_2+\alpha b_2 < (1-\alpha) {a'}_2+\alpha {b'}_2$ we obtain that $
 \terv{a}{b}{2} \preceq_{\alpha \beta} \terv{a'}{b'}{2}
 $.

 Moreover, as ${a'}_2 + \frac{\al}{1-\al}\left(\sqrt[m]{b^m+{b'}_2^m} - \sqrt[m]{b^m + b_2^m}\right) < a_2$ and ${a'}_2>0$, we also have that $a_2>0$.

On the other hand,
$$m\textrm{PM}([0,b],\terv{a}{b}{2}) = (1-\al)\sqrt[m]{\frac{0^m+a_2^m}{2}} + \al\sqrt[m]{\frac{b^m+b_2^m}{2}} = (1-\alpha) \frac{a_2}{\sqrt[m]{2}} + \alpha \frac{\sqrt[m]{b^m+b_2^m}}{\sqrt[m]{2}}$$
and
$$m\textrm{PM}([0,b],\terv{a'}{b'}{2})  = (1-\al)\sqrt[m]{\frac{0^m+{a'}_2^m}{2}} + \al\sqrt[m]{\frac{b^m+{b'}_2^m}{2}}= (1-\alpha) \frac{{a'}_2}{\sqrt[m]{2}} + \alpha \frac{\sqrt[m]{b^m+{b'}_2^m}}{\sqrt[m]{2}}$$
 Thus, since $\alpha\in (0,1)$ and ${a'}_2 + \frac{\al}{1-\al}\left(\sqrt[m]{b^m+{b'}_2^m} - \sqrt[m]{b^m + b_2^m}\right) < a_2$, we obtain that
$$m\textrm{PM}([0,b],\terv{a'}{b'}{2}) \prec_{\alpha\beta} m\textrm{PM}([0,b],\terv{a}{b}{2}),$$
 that is, $m$PM is not $\preceq_{\alpha \beta}$-increasing.

\end{example}

The counterexample is valid in particular for $\alpha=\frac{1}{2}, \beta > \alpha$, this is, we have as a particular case that the $m$PM operator is not a $\xuyor$-average.

A final remark regarding the $m$th power mean is that the generalized power mean of order $1$ corresponds to the arithmetic mean for intervals, following the standard terminology used in Statistics.

The problematic behavior in the even case is resolved  when considering intervals composed of non-negative numbers.

\begin{proposition}\label{prop:mPMseparable}
The operator $m$PM$:\LLposit^n\to\LL$ is separable with component functions $f=g=mpm$ for any $m \in \naturales$.
\end{proposition}
\begin{proof}
Observe that if $[\unda,\ova]\in\LLposit$, then $[\unda,\ova]^m=[\unda^m,\ova^m]$ for any $m \ge 2$ and this implies that
\[
m\mbox{PM}(\colec)=\left[\sqrt[m]{\frac{\sum_{i=1}^n \unda_i^m}{n}},\sqrt[m]{\frac{\sum_{i=1}^n \ova_i^m}{n}}\right].
\]
\end{proof}

\begin{proposition}
    Let $m\in \naturales$.  The operator $m$PM$:\LLposit^n\to\LL$ is a
 $\lator$-average,  $\lexunor$-average and
 $\lexdosor$-average.
\end{proposition}
\begin{proof} As proven in Proposition~\ref{prop:mPMseparable}, $mPM$ is separable with $f=g=m\mbox{pm}$.
The operator $m\mbox{pm}:\reales^+ \to \reales$ is an idempotent and strictly increasing function. Therefore the proof follows from Propositions~\ref{prop:idempotency_h<->fg},~\ref{prop:monotonicity_lator} and~\ref{prop:monotonicity_lexunor}.
\end{proof}

\subsection{The median}

We have already examined several location measures for intervals such as the maximum and minimum (when defined), as well as the supremum and infimum and central tendency measures, including power means, with the arithmetic mean as a particular case. Another classical average function for real numbers is the median. We now proceed to formalize the concept of the median for intervals.

 As it happens with other measures of position, as the minimum and maximum, if we consider an ontic point of view, intervals should be ordered before determining the median. Total orders are the ones that guarantee that we can arrange any finite set of intervals. In this case the natural definition of median is the following one.




\begin{definition}
  Given a total order $\preceq_{to}$ on $\LL$ and a collection of intervals $\{a_1,\dots,a_n\}$ also in $\LL$, we denote $\{a_{\sigma(1)},\dots, a_{\sigma(n)}\}$ the reorganization of the intervals that satisfies
  $ a_{\sigma(1)} \preceq_{to} a_{\sigma(2)} \preceq_{to} \dots \preceq_{to}
  a_{\sigma(n)}$.
  Then the median according to $\preceq_{to}$, denoted $Median_{\preceq_{to}}$ of these intervals is
  \begin{itemize}
    \item $a_{\sigma(\frac{n+1}{2})}$ if $n$ is odd,
    \item $\left[\frac{\underline{a}_{\sigma(\frac{n}{2})}+ \underline{a}_{\sigma(\frac{n}{2}+1)}}{2},\frac{\overline{a}_{\sigma(\frac{n}{2})}+ \overline{a}_{\sigma(\frac{n}{2}+1)}}{2}\right]$ if $n$ is even.
  \end{itemize}
\end{definition}

By definition, it is immediate to prove the following proposition.

\begin{proposition}
  For any total order $\preceq_{to}$ on $\LL$,   $Median_{\preceq_{to}}$ is a $\preceq_{to}$-average.
\end{proposition}

However, if we consider the epistemic point of view and we think of the intervals as ranges of possible values for the actual decision maker choice, the median should be the interval that contains all the possible medians and in this case it should have as bounds the median of the lower bounds of the original intervals and as upper bound the median of the upper bounds of the original intervals.  We then get the following definition of median.

\begin{definition}
  Given a set of intervals $\colec \in \LL$, we define its separable median as follows:
  $$
  Median^*(\colec)=\left[Me(\unda_1,\dots,\unda_n),Me(\ova_1,\dots,\ova_n)\right]
  $$
  where $Me$ denotes the median of a set of values in $\reales$.
\end{definition}

Observe that since $Me$ is monotone, it always holds that $Me(\unda_1,\dots,\unda_n) \le Me(\ova_1,\dots,\ova_n)$ and so, the function $Median$ is well defined. Observe also that this definition does not depend on any order defined in $\LL$, but only on the natural order in $\reales$. It orders the extremal points of the intervals to obtain the outcome. Therefore, its behaviour is appropriate for the lattice order.

\begin{proposition}
 $Median^*$ is a $\preceq_{Lo}$-average.
\end{proposition}

\begin{proof}
Idempotency is evident in this case.

Moreover,  it is a separable function, where the two component functions are the median defined for numbers. Since this is a monotone function, it follows from Proposition~\ref{prop:monotonicity_lator} that the Median function for intervals is $\lator$-increasing.
\end{proof}

In addition, if the departing intervals are ordered according to the lattice order, then $Median_{\preceq_{Lo}}$ is in fact $Median^*$

\begin{lemma}\label{lem:eqMedianLo}
    If the intervals $a_1,\dots,a_n$ with $a_i \in \LL$ for all $i$, satisfy that there exists a permutation $\sigma$ of the $n$ indexes such that $a_{\sigma(1)}\lator a_{\sigma{2}} \lator \dots \lator a_{\sigma(n)}$. It holds that $Median_{\preceq_{Lo}}(a_1,\dots,a_n)=Median^*(a_1,\dots,a_n)$.
\end{lemma}
\begin{proof}
    Consider $n$ an odd number, then $Median_{\preceq_{Lo}}(a_1,\dots,a_n)=a_{\sigma(\frac{n+1}{2})}$. On the other hand, if all the intervals are ordered according to the lattice order, it holds that $\unda_{\sigma(1)} \le \dots \le \unda_{\sigma(n)}$ and $\ova_{\sigma(1)} \le \dots \le \ova_{\sigma(n)}$, whereas $Me(\unda_{\sigma(1)}, \dots, \unda_{\sigma(n)})=\unda_{\sigma(\frac{n+1}{2})}$ and $Me(\ova_{\sigma(1)}, \dots, \ova_{\sigma(n)})=\ova_{\sigma(\frac{n+1}{2})}$. It follows that also $Median^*(a_1,\dots,a_n)=(\unda_{\sigma(\frac{n+1}{2})},\ova_{\sigma(\frac{n+1}{2})})$
\end{proof}

However, this is not the case for the main examples of admissible orders.
 Since the median for real numbers is not strictly increasing, we cannot apply Proposition~\ref{prop:monotonicity_lexunor}. In fact, it is  not $\preab$-increasing for any $\alpha \in [0,1]$ with $\beta \neq \alpha$, as we can see in the following examples.

\begin{example}
\begin{itemize}
\item The case $\alpha=0$.

Observe that $[0.2,0.7] \lexunor [0.3,0.5]$, however,
$$
\begin{array}{rl}
Median^*([0,0.4],[0.2,0.7],[0.2,0.6]) & =[0.2,0.6] \not\lexunor [0.2,0.5]\\[.4cm]
 & =Median^*([0,0.4],[0.3,0.5],[0.2,0.6])\,.
 \end{array}
$$
Thus, the separable median is not $\lexunor$-increasing.

\item The case $\alpha=1$.
Analogously, the separable median is not $\lexdosor$-increasing. Observe that $[0.3,0.6]\lexdosor[0.2,0.7]$, but
$$
\begin{array}{rl}
Median^*([0,0.2],[0.3,0.6],[0.4,0.5]) & =[0.3,0.5] \not\lexdosor  [0.2,0.5]\\[.4cm]
& =Median^*([0,0.2],[0.2,0.7],[0.4,0.5])\,.
\end{array}
$$

\item The case  $\alpha \in (0,1)$ with $\beta \neq \alpha$.

We will consider that $\beta>\alpha$.
 Fix $\alpha \in (0,1)$ and take $x_1, x_2 \in [0,1]$ such that $x_1 < x_2$. As $x_2\le 1$, we have that $1-(1-\alpha)x_2 = (1-x_2)+\alpha x_2\ge \alpha x_2$  and so we can consider $y\in\mathbb{R}$ such that  $\alpha x_2 \le y \le 1-(1-\alpha)x_2$.

 Now take the intervals $a=[y-\alpha x_1, y + (1-\alpha)x_1]$ and $b=[y-\alpha x_2, y + (1-\alpha)x_2]$. Then
 $K_{\alpha}(a)=y=K_{\alpha}(b)$ and
$
K_{\beta}(a)=y+(\beta-\alpha)x_1<y+(\beta-\alpha)x_2=K_{\beta}(b)
$, which implies that $a \preab b$.

To simplify the notation, let   $a:= [\underline{a},\overline{a}] $ and $b:= [\underline{b},\overline{b}]$-
 Since $\alpha<1$ and $x_1<x_2$, we can consider an element $\epsilon$ in the interval $(0,(1-\alpha)(x_2-x_1))$.   Then,
$
Median([0,0],a,[\underline{a},\overline{a}+\epsilon])=a\,.
$ and
$
Median([0,0],b,[\underline{a},\overline{a}+\epsilon])=[\underline{b},\overline{a}+\epsilon]\,.
$

On the other hand, it holds that
$$
\begin{array}{ll}
K_{\alpha}([\underline{b},\overline{a}+\epsilon]) & = K_{\alpha}([y-\alpha x_2,y+ (1-\alpha)x_1 + \epsilon]) \\[.2cm] & = y-\alpha x_2 + \alpha (y+(1-\alpha)x_1 + \epsilon - y+\alpha x_2)\\[.2cm]
 & = y-\alpha x_2 + \alpha (x_1 + \alpha (x_2-x_1)+\epsilon) = y + \alpha (\epsilon-(1-\alpha)(x_2-x_1))
\end{array}
$$
And, since $\epsilon <(1-\alpha)(x_2-x_1)$, we have that
$y + \alpha (\epsilon-(1-\alpha)(x_2-x_1)) < y$.

Moreover,  $K_{\alpha}(a)=y$ and it follows that
$$
Median([0,0],b,[\underline{a},\overline{a}+\epsilon]) \prec_{\alpha\beta} Median([0,0],a,[\underline{a},\overline{a}+\epsilon])
$$
which implies that the separable median is not $\preab$-increasing for $\alpha< \beta$.

The counterexample for $\beta<\alpha$ is analogous and therefore not included here.
\end{itemize}

\end{example}

In the next tables we summarize the main  results obtained in this paper.

    \begin{tabular}{ccccc}
    \multicolumn{1}{c}{   } & \multicolumn{4}{c}{average with respect to the order}\\
             \hline
        Operator & $\lator$-average & $\lexunor$ & $\lexdosor$  &  \begin{tabular}{c}
            \\
               $\preceq_{\alpha\beta}$   \\
                 \scriptsize{$\alpha,\beta \in (0,1)$}
            \end{tabular} \\
        \hline\hline
       Min &   & \cmark & \cmark & \cmark  \\
        \hline
        Max &   & \cmark & \cmark& \cmark \\
\hline
        $\inflo$ & \cmark & \xmark & \xmark & \xmark  \\
        \hline
        $\suplo$ & \cmark & \xmark & \xmark& \xmark \\
        \hline\hline
             $mPM$, {$m \ge 3$, $m$ odd}
        & \cmark & \cmark & \cmark & \xmark \\
         \hline
             $mPM$,
            {$m$ even} & \xmark & \xmark & \xmark & \xmark  \\
        \hline
        Mean & \cmark & \cmark & \cmark & \cmark  \\
            \hline\hline
            Median &   & \cmark & \cmark & \cmark \\
            \hline
            Sep. Median & \cmark & \xmark & \xmark & \xmark \\
            \hline
    \end{tabular}

In the case of the mPM operator, depending on the value of m and the domain of definition, the result vary. If m is an odd number greater than 2 and the operator is defined over $\LL^n$, we get the next table.

\begin{center}
\begin{tabular}{cccc}
            Operator & $\lator$-average & $\preceq_{\alpha+}$-average & $\preceq_{\alpha-}$-average \\
            \hline
            $mPM$ & \cmark & \xmark & \xmark \\
            \hline
        \end{tabular}
\end{center}

If m is an even number greater than one and the operator is defined over $\LL^n$, we get the next table.

\begin{center}
\begin{tabular}{cccc}
            Operator & $\lator$-average & $\preceq_{\alpha+}$-average & $\preceq_{\alpha-}$-average \\
            \hline
            $mPM$ & \xmark & \xmark & \xmark \\
            \hline
        \end{tabular}
        \end{center}

On the other hand, if we restrict the operator to positive intervals, that is, if it is defined over $\LLposit^n$ its behaviour does not depend on the parity of $m$ and, for every $m\geq 2$ we get that

\begin{center}
\begin{tabular}{ccccc}
    \multicolumn{1}{c}{   } & \multicolumn{4}{c}{average with respect to the order}\\
             \hline
            \begin{tabular}{c}
                 Operator \\
                  over $\LLposit^n$
            \end{tabular}& $\lator$ & $\lexunor$ & $\lexdosor$ & \begin{tabular}{c}
      \\
               $\preceq_{\alpha\beta}$   \\
                 \scriptsize{$\alpha,\beta \in (0,1)$}
            \end{tabular}   \\
            \hline
$mPM, m\ge 2$ & \cmark & \cmark & \cmark & \xmark\\
            \bottomrule
        \end{tabular}
        \end{center}



\section{Application Example: Aggregation of Air Pollution Interval Data}\label{sec:example}


In this section we show a small example of application of the aggregation functions presented in the paper.

Suppose we want to know the concentration of a contaminant in a river, but we only have a sensor that measures with some imprecision. The sensor indicates that the concentration lies between 12 and 15 mg/L, due to its limited resolution and measurement noise. Here, the interval \([12, 15]\) mg/L does not reflect a real variability in the water concentration (which could be constant at the place and time of measurement), but rather our uncertainty about the exact contaminant value. The interval represents a lack of precise knowledge   that is, the true value lies somewhere inside the interval, but we do not know where.

In this example, we consider daily PM2.5 air pollution concentrations (measured in \(\mu g/m^3\)) represented as intervals to reflect measurement variability and environmental fluctuations. The data correspond to 10 consecutive days, with each day s pollution concentration expressed as an interval \([\unda_j, \ova_j]\), where \(\unda_j\) and \(\ova_j\) denote the lower and upper bounds, respectively.

\begin{table}[h!]
\centering
\begin{tabular}{ccc}
\hline
Day & Lower Bound \(\unda_j\) & Upper Bound \(\ova_j\) \\
\hline
1  & 12.1 & 15.3 \\
2  & 11.7 & 16.1 \\
3  & 13.5 & 17.2 \\
4  & 12.9 & 16.8 \\
5  & 13.2 & 15.7 \\
6  & 11.9 & 15.0 \\
7  & 12.8 & 16.4 \\
8  & 13.0 & 15.8 \\
9  & 12.3 & 16.0 \\
10 & 11.8 & 14.9 \\
\hline
\end{tabular}
\caption{Daily PM2.5 concentration intervals (\(\mu g/m^3\)) from a public dataset \cite{EPA_PM25_dataset}.}
\label{tab:pm25data}
\end{table}

Using these intervals, we compute various interval-valued aggregations to summarize the data, as follows:

\begin{itemize}
    \item \textbf{Arithmetic Mean:}
    \[
    \begin{array}{rl}
    1PM([\unda_1,\ova_1],\dots,[\unda_{10},\ova_{10}]) & = \left[ \displaystyle\frac{1}{10}\sum_{j=1}^{10} \unda_j, \quad \frac{1}{10}\displaystyle\sum_{j=1}^{10} \ova_j \right]\\[.4cm]
    & =\left[ \displaystyle\frac{123.2}{10}, \quad \displaystyle\frac{159.2}{10} \right] = [12.32, 15.92].
    \end{array}
    \]

    \item \textbf{Infimum (Optimistic aggregation):}
    \[
    \inflo([\unda_1,\ova_1],\dots,[\unda_{10},\ova_{10}])= [\min_j \unda_j, \min_j \ova_j] = [11.7, 14.9].
    \]

    \item \textbf{Supremum (Pesimistic aggregation):}
    \[
    \suplo([\unda_1,\ova_1],\dots,[\unda_{10},\ova_{10}])=[\max_j \unda_j, \max_j \ova_j] = [13.5, 17.2].
    \]

    \item \textbf{Median:}
    \[
     Median^*([\unda_1,\ova_1],\dots,[\unda_{10},\ova_{10}])=\left[\mathrm{median}(\unda_j), \mathrm{median}(\ova_j)\right] = [12.3, 15.85].
    \]

    \item \textbf{Root Mean Square (RMS):}
    \[
    \begin{array}{rl}
    2PM([\unda_1,\ova_1],\dots,[\unda_{10},\ova_{10}]) & =\left[\sqrt{\frac{1}{10}\displaystyle\sum_{j=1}^{10} \unda_j^2}, \enskip \sqrt{\frac{1}{10}\displaystyle\sum_{j=1}^{10} \ova_j^2}\right]\\[.5cm] & = \left[\displaystyle\sqrt{\frac{1527.34}{10}}, \enskip \displaystyle\sqrt{\frac{2559.91}{10}}\right]  = [12.38, 16.03].
    \end{array}
    \]
\end{itemize}

Here, the arithmetic mean provides a central tendency interval by averaging the lower and upper bounds separately. The infimum and supremum represent optimistic and pesimistic bounds, respectively, capturing the minimal and maximal observed pollution values across days. The median uses the central value of the sorted lower and upper bounds, offering a robust summary. Finally, the root mean square (RMS) aggregates the quadratic means of the bounds, giving more weight to larger values and reflecting variability in the data.

This example illustrates the practical use of interval aggregation methods in environmental data analysis, where uncertainty and variability naturally lead to interval-valued observations.

\section{Conclusion}\label{sec:conclusion}

In this paper, we have addressed the formal definition of average functions for interval data, extending classical aggregation theory to settings where information is inherently imprecise or uncertain. We have shown that the choice of order on the set of intervals plays a crucial role in determining the properties and applicability of interval aggregation operators. In particular, while the lattice order is the most intuitive, its lack of totality motivates the consideration of alternative total orders, each with its own implications for monotonicity and idempotency.

We have provided rigorous definitions of average functions in the interval context and have characterized the behavior of classical aggregation operators, such as power means, minimum, maximum, and median, when extended to intervals. Our analysis reveals that not all classical operators retain their desirable properties under all admissible orders, highlighting the importance of carefully selecting the order structure according to the application at hand.

The framework developed in this work lays the foundation for systematic aggregation of interval-valued information, with potential applications in statistics, decision making, and any field where data imprecision is present. Future research may focus on further generalizations, computational aspects, and the development of new aggregation operators tailored to specific types of interval data and practical requirements.


\section*{Funding}
The research reported on this contribution was funded by the Spanish Ministry of Science and Innovation through projects PID2022-139886NB-I00 and PID2023-146621OB-C21 and by the
Asturian Agency for Science, Business Competitiveness and Innovation (SEKUENS)
project
SEK-25-GRU-GIC-24-018. Susana Montes and Susana Diaz-Vazquez also acknowledge the funding received from the European Project SCIMIN-CRM - 101177746 - GAP-101177746.


\bibliography{biblio}

@inproceedings{Atanassov83,
                    title={Intuitionistic fuzzy sets},
                    author={Atanassov, K.},
                    booktitle={VIIth ITKR Session},
                    pages={1684--1697},
                    year={1983},
                    address={Sofia,Bulgaria},
                    organization={Deposited in the Central Science and Technology Library of the Bulgarian Academy of Sciences}
                  }

@book{beliakov2007aggregation,
  title={Aggregation functions: A guide for practitioners},
  author={Beliakov, Gleb and Pradera, Ana and Calvo, Tomasa and others},
  volume={221},
  year={2007},
address={Heidelberg},
  publisher={Springer}
}

@book{beliakov2016practical,
  title={A practical guide to averaging functions},
  author={Beliakov, Gleb and Bustince Sola, Humberto and Calvo S{\'a}nchez, Tomasa},
  volume={329},
  year={2016},
address={Switzerland},
  publisher={Springer}
}

@article{Bustince,
author = {H. Bustince and J. Fernandez and A. Kolesárová and R. Mesiar},
title = {Generation of linear orders for intervals by means of aggregation functions},
journal= {Fuzzy Sets and Systems 86},
volumen="220",
doi = {https://doi.org/10.1016/j.fss.2012.07.015},
pages="69-77",
year = {2013}
}

@book{cauchy1821,
  title={Cours d'analyse de l'{\'E}cole Royale Polytechnique},
  author={Cauchy, Augustin Louis Baron},
  volume={1},
  year={1821},
address={Paris},
  publisher={Imprimerie royale}
}

@article{COUSO20141502,
title = {Statistical reasoning with set-valued information: Ontic vs. epistemic views},
journal = {International Journal of Approximate Reasoning},
volume = {55},
number = {7},
pages = {1502-1518},
year = {2014},
note = {Special issue: Harnessing the information contained in low-quality data sources},
issn = {0888-613X},
doi = {https://doi.org/10.1016/j.ijar.2013.07.002},
url = {https://www.sciencedirect.com/science/article/pii/S0888613X13001618},
author = {Couso, Inés and Dubois, Didier}
}

@book{davey2002introduction,
  title={Introduction to Lattices and Order},
  author={Davey, B.A. and Priestley, H.A.},
  isbn={9780521784511},
  lccn={2001043910},
  series={Cambridge mathematical textbooks},
  doi={https://books.google.es/books?id=vVVTxeuiyvQC},
  year={2002},
address={Cambridge},
  publisher={Cambridge University Press}
}

@online{EPA_PM25_dataset,
  title     = {Air Quality System Data Mart},
  author    = {{United~States~Environmental~Protection~Agency}},
  year      = {2025},
  url       = {https://www.epa.gov/aqs},
  urldate   = {2025-06}
}

@article{Fishburn1987,
author = {P. Fishburn},
title = {Interval Ordenings},
journal= {Wiley, New York},
year = {1987}
}

@book{grabisch2009,
place={Cambridge},
series={Encyclopedia of Mathematics and its Applications},
title={Aggregation Functions},
DOI={10.1017/CBO9781139644150},
publisher={Cambridge University Press},
author={Grabisch, Michel and Marichal, Jean-Luc and Mesiar, Radko and Pap, Endre},
year={2009},
address={Cambridge},
collection={Encyclopedia of Mathematics and its Applications}}

@article{HroBriFil2017,
  title={Exploratory data analysis for interval compositional data},
  author={Hron, Karel and Brito, Paula and Filzmoser, Peter},
  journal={Advances in Data Analysis and Classification},
  volume={11},
  number={2},
  pages={223--241},
  year={2017},
  publisher={Springer}
}

@book{neggers1998basic,
  title={Basic Posets},
  author={Neggers, J. and Kim, H.S.},
  isbn={9789810235895},
  series={G - Reference,Information and Interdisciplinary Subjects Series},
  doi={https://books.google.es/books?id=-ip3-wejeR8C},
address={Singapore},
  year={1998},
  publisher={World Scientific}
}

@article{OliAzePacVal2022,
  title={Association measures for interval variables},
  author={Oliveira, M Ros{\'a}rio and Azeitona, Margarida and Pacheco, Ant{\'o}nio and Valadas, Rui},
  journal={Advances in Data Analysis and Classification},
  volume={16},
  number={3},
  pages={491--520},
  year={2022},
  publisher={Springer}
}

@book{taylor1999practical,
  title={Practical Foundations of Mathematics},
  author={Taylor, P.},
  number={v. 59},
  isbn={9780521631075},
  lccn={98039472},
  series={Cambridge Studies in Advanced Mathematics},
  doi={https://books.google.es/books?id=iSCqyNgzamcC},
  year={1999},
address={Cambridge},
  publisher={Cambridge University Press}
}

@article{Xu2006,
author = {Xu, Z.  and  Yager,R.~R.},
title = {Some geometric aggregation operators based on intuitionistic fuzzy sets},
journal = {International Journal of General Systems},
volume = {35},
number = {4},
pages = {417-433},
year  = {2006},
publisher = {Taylor \& Francis},
doi = {10.1080/03081070600574353},
URL = {https://doi.org/10.1080/03081070600574353}
}

@article{archibald,
  title={Mathematics before the Greeks},
  author={Archibald, Raymond Clare},
  journal={Science},
  volume={71},
  number={1831},
  pages={109--121},
  year={1930},
  publisher={American Association for the Advancement of Science}
}

@article{zadeh1965fuzzy,
  title={Fuzzy sets},
  author={Zadeh, Lotfi A},
  journal={Information and Control},
  year={1965}
}

@article{yager2004owa,
  author    = {Ronald R. Yager},
  title     = {OWA aggregation over a continuous interval argument with applications to decision making},
  journal   = {IEEE Transactions on Systems, Man, and Cybernetics. Part B (Cybernetics)},
  volume    = {34},
  number    = {5},
  pages     = {1952--1963},
  year      = {2004},
}

@book{moore1966,
  author    = {Ramon E. Moore},
  title     = {Interval Analysis},
  year      = {1966},
  publisher = {Prentice-Hall},
  address   = {Englewood Cliffs, NJ},
}

@book{dubois1988,
  author    = {Didier Dubois and Henri Prade},
  title     = {Possibility Theory: An Approach to Computerized Processing of Uncertainty},
  year      = {1988},
  publisher = {Plenum Press},
  address   = {New York},
}

@phdthesis{sambuc1975,
  author       = {Jean-Pierre Sambuc},
  title        = {Fonctions $\phi$-floues. Application à l'aide au diagnostic en pathologie thyroidienne},
  year         = {1975},
  school       = {Université de Marseille},
  note         = {Thèse de troisième cycle},
  address      = {Marseille, France}
}

@article{diazvazquez2024new,
  title={A new family of aggregation functions for intervals},
  author={Diaz-Vazquez, Susana and Torres-Manzanera, Emilio and Rico, Noelia and Mesiar, Radko and Rodriguez-Martinez, Iosu and Lafuente, Julio and Diaz, Irene and Montes, Susana and Bustince, Humberto},
  journal={Computational and Applied Mathematics},
  doi={https://doi.org/10.1007/s40314-023-02525-1},
  volume={43},
  number={1},
  pages={17},
  year={2024},
  publisher={Springer}
}

@article{fumanal2021,
  title={Interval-valued aggregation functions based on moderate deviations applied to motor-imagery-based brain--computer interface},
  author={Fumanal-Idocin, Javier and Tak{\'a}{\v{c}}, Zdenko and Fern{\'a}ndez, Javier and Sanz, Jos{\'e} Antonio and Goyena, Harkaitz and Lin, Ching-Teng and Wang, Yu-Kai and Bustince, Humberto},
  journal={IEEE Transactions on Fuzzy Systems},
  volume={30},
  number={7},
  pages={2706--2720},
  year={2021},
doi={https://doi.org/10.1109/TFUZZ.2021.3092824},
  publisher={IEEE}
}

@article{WilkinBeliakov2022,
  title={Robust aggregation of compositional and interval-valued data: The mode on the unit simplex},
  author={Wilkin, Tim and Beliakov, Gleb},
  journal={Fuzzy Sets and Systems},
  volume={446},
  pages={124--143},
  year={2022},
doi = {https://doi.org/10.1016/j.fss.2021.01.007},
  publisher={Elsevier}
}

@article{ZhangElastic2021,
author = {Honggang Zhang and Jingyong Su and Linlin Tang and Anuj Srivastava},
title = {Elastic statistical analysis of interval-valued time series},
journal = {Journal of Applied Statistics},
volume = {50},
number = {1},
pages = {60--85},
year = {2023},
publisher = {Taylor \& Francis},
doi = {10.1080/02664763.2021.1981257},
note ={PMID: 36530776},
URL = {https://doi.org/10.1080/02664763.2021.1981257},
eprint = {https://doi.org/10.1080/02664763.2021.1981257}
}

@article{sun2016,
  title={Linear regression with interval-valued data},
  author={Sun, Yan},
  journal={Wiley Interdisciplinary Reviews: Computational Statistics},
  volume={8},
  number={1},
  pages={54--60},
  year={2016},
  publisher={Wiley Online Library}
}

\end{document}